\documentclass[a4paper,10pt]{amsart}
\usepackage[utf8]{inputenc}
\usepackage{amsthm}
\usepackage{comment}
\usepackage{amsmath}
\usepackage{bm}
\usepackage{amsfonts}
\usepackage{amssymb, enumitem}
\usepackage{mathtools}
\usepackage{appendix}
\usepackage{mathrsfs}
\usepackage{setspace}
\usepackage{xcolor}
\usepackage{todonotes}
\usepackage{thmtools} 
\usepackage[foot]{amsaddr}
\usepackage[pdfdisplaydoctitle,colorlinks,breaklinks,urlcolor=blue,linkcolor=blue,citecolor=blue]{hyperref} 
\usepackage[nameinlink,capitalise]{cleveref}

\newcommand{\C}{\mathbb{C}}

\newcommand{\HH}{\mathcal{H}}

\newcommand{\N}{\mathbb{N}}

\newcommand{\R}{\mathbb{R}}

\newcommand{\T}{\mathbb{T}}

\let\div\relax

\DeclareMathOperator{\div}{div}

\renewcommand{\epsilon}{\varepsilon}
\renewcommand{\setminus}{\smallsetminus}

\newcommand{\one}{\bm{1}}

\newtheorem{theorem}{Theorem}[section]
\newtheorem{definition}[theorem]{Definition}
\newtheorem{hypothesis}[theorem]{Hypothesis}
\newtheorem{corollary}[theorem]{Corollary}
\newtheorem{lemma}[theorem]{Lemma}
\newtheorem{proposition}[theorem]{Proposition}

\theoremstyle{remark}
\newtheorem{remark}[theorem]{Remark}

\numberwithin{equation}{section}

\newcommand{\angH}{\omega_{H^{\infty}}}
\newcommand{\p}{\mathbb{P}}
\newcommand{\Tor}{\mathbb{T}}
\newcommand{\Do}{\mathsf{D}}
\newcommand{\Dom}{\mathcal{O}}
\newcommand{\Ld}{L_{q,{\rm D}}}
\newcommand{\Lr}{L_{q,{\rm R}}}
\newcommand{\Ldd}{\ell_{q,{\rm D}}}
\newcommand{\Lrd}{\ell_{q,{\rm R}}}
\newcommand{\Lt}{\ell_{q,{\rm P}}}
\newcommand{\rsec}{\mathcal{R}}
\newcommand{\Ls}{\mathbb{L}}
\newcommand{\Ws}{\mathbb{W}}
\newcommand{\Hs}{\mathbb{H}}
\newcommand{\dd}{\mathrm{d}}
\newcommand{\embed}{\hookrightarrow}
\newcommand{\NM}{\mathcal{N}}
\allowdisplaybreaks

\title[Navier-Stokes Equations with Boundary Noise]{Global well-Posedness and Interior Regularity\\ of 2D Navier-Stokes Equations with\\ Stochastic Boundary Conditions}
\author{Antonio Agresti}
\address{Institute of Science and Technology Austria (ISTA), Am Campus 1, 3400 Klosterneuburg, Austria}
\email{\href{mailto:antonio.agresti92 at gmail.com}{antonio.agresti92 at gmail.com}}
\author{Eliseo Luongo}
\address{Scuola Normale Superiore, Piazza dei Cavalieri, 7, 56126 Pisa, Italia}
\email{\href{mailto:eliseo.luongo at sns.it}{eliseo.luongo at sns.it}}

\thanks{The first author has received funding from the European Research Council (ERC) under the Eu\-ropean Union’s Horizon 2020 research and innovation programme (grant agreement No 948819) \includegraphics[height=0.4cm]{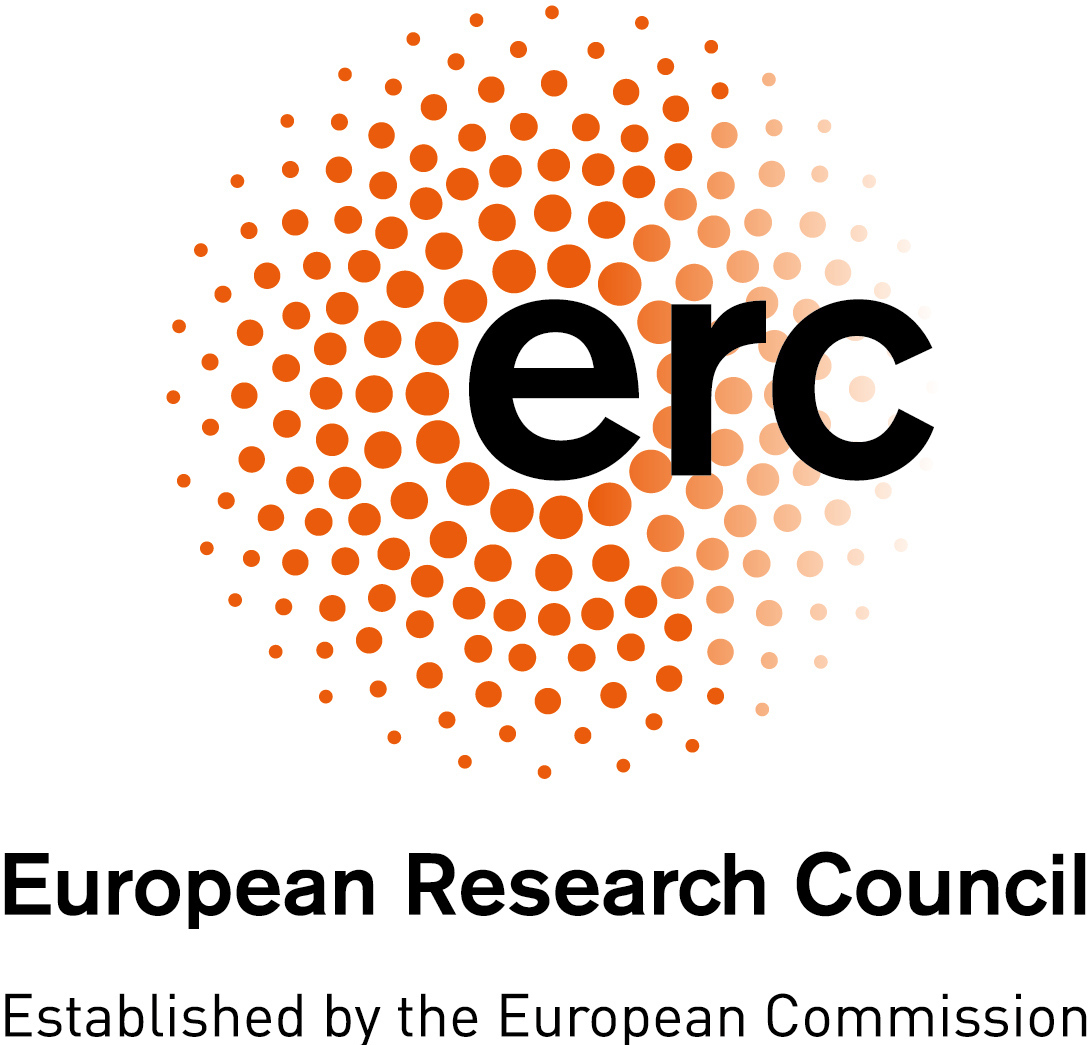}\,\includegraphics[height=0.4cm]{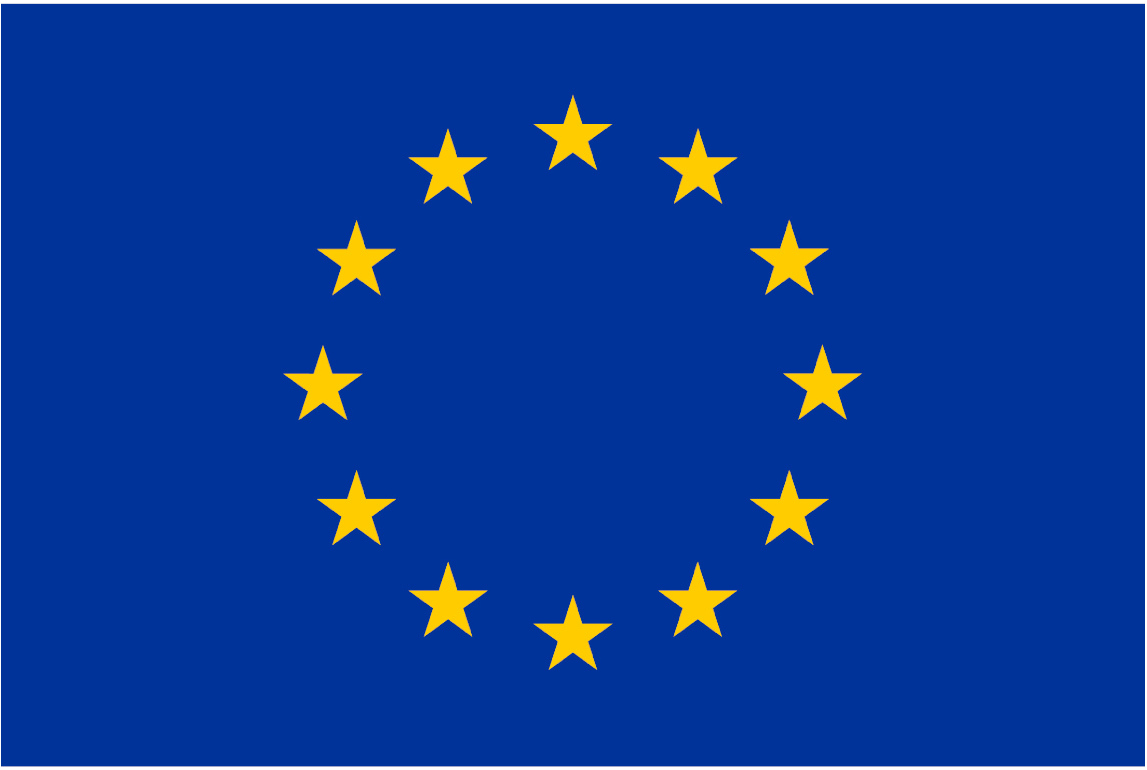}.}

\date\today

\begin{document}

\begin{abstract}
The paper is devoted to the analysis of the global well-posedness and the interior regularity of the 2D Navier-Stokes equations with inhomogeneous stochastic boundary conditions. The noise, white in time and coloured in space, can be interpreted as the physical law describing the driving mechanism on the atmosphere-ocean interface, i.e. as a balance of the shear stress of the ocean and the horizontal wind force.
\end{abstract}

\maketitle

\tableofcontents

\section{Introduction}\label{sec:introduction}
Partial differential equations with boundary noise have been introduced by Da Prato and Zabczyck in the seminal paper \cite{da1993evolution}. They showed that, also in the one dimensional case, the solutions of the heat equation with white noise Dirichlet or Neumann Boundary conditions have low regularity compared to the case of noise diffused inside the domain. In particular, in the case of Dirichlet boundary conditions the solution is only a distribution. Some improvements in the analysis of the interior regularity of the solutions of these problems and some nonlinear variants have been obtained exploiting specific properties of the heat kernel and of suitable nonlinearities. For some results in this direction we refer to \cite{alos2002stochastic,Zanella,debussche2007optimal,fabbri2009lq,Goldys23}. All these issues make the problem of treating non-linear partial differential equations with boundary noise coming from fluid dynamical models an, almost untouched, field of open problems. 

Throughout the manuscript we fix a finite time horizon $T>0$. Let $a>0$, $\Dom=\mathbb{T}\times(0,a)$ and let $\mathbb{T}$ be the one dimensional torus. Finally, 
we denote by 
\begin{equation}
\Gamma_b=\mathbb{T}\times\{0\}\quad\text{ and }\quad\Gamma_u=\mathbb{T}\times\{a\},
\end{equation} 
the bottom and the upper part of the boundary of $\Dom$, respectively. 

In this paper we are interested in the global well-posedness and the interior regularity of the 2D Navier-Stokes equations with boundary noise for the unknown velocity field $u(t,\omega,x,z)=(u_1,u_2):\R_+\times \Omega\times \Dom \to \R^2$, formally written as
\begin{equation}\label{Intro equation}
\left\{
\begin{aligned}
\partial_{t}u+u\cdot\nabla u+\nabla P  & =\Delta u, \qquad & \text{ on }&(0,T)\times\Dom, \\
\operatorname{div}u  & =0, & \text{ on }&(0,T)\times\Dom,  \\ 
u    & = 0,& \text{ on }&(0,T)\times\Gamma_b,\\
\partial_{z}u_1    & =h_b \dot{W}_{\mathcal{H}}, & \text{ on }&(0,T)\times\Gamma_u,\\
u_2   & =0, & \text{ on }&(0,T)\times\Gamma_u, \\ 
u(  0)    & =u_0, & \text{ on }&\Dom, 
\end{aligned}    \right.
\end{equation}
where $\nabla u=(\partial_j u_i)_{i,j=1}^2$, $W_{\mathcal{H}}(t)$ is a $\mathcal{H}$-cylindrical Brownian motion and $h_b(t,x)$ is a sufficiently regular forcing term; we refer to \Cref{sec:main results} below for the  the relevant assumptions and definitions.
To the best of our knowledge this is the first instance of a \emph{global} well-posedness result for a fluid dynamical system driven by stochastic white in time boundary conditions. We refer to \cite{bessaih2014homogenization,bessaih2016homogenization} for some homogenization results in the case of Navier-Stokes equations with dynamic boundary conditions driven by a stochastic forcing and to \cite{binz2020primitive} for the local analysis of the three dimensional primitive equations with boundary noise. Finally, we refer to \cite{dalibard2009asymptotic,dalibard2009mathematical} for some limit behaviors of the model \eqref{Intro equation} with $h_b \dot{W}_{\mathcal{H}}$ replaced by a highly oscillating and regular stationary random field.


Following the books by Pedlosky \cite{pedlosky1996ocean,pedlosky2013geophysical} and Gill \cite{gill1982atmosphere}, the model \eqref{Intro equation} is a good idealization of the velocity of the fluid in the ocean. 
In this scenario, the domain $\Dom=\T \times (0,a)$ can be considered a vertical slice of the ocean with depth $a>0$ and we should interpret $u_1$ (resp. $u_2$) as the horizontal (resp. vertical) component of the velocity field $u$.
Indeed even if, in principle, one should consider a free surface, instead of $\Gamma_u=\T\times \{a\}$, depending on the time, the approximation of such surface as independent of the time, although highly unrealistic, is justified by the fact that the behavior of the fluid around the surface is in general very turbulent. Hence, as emphasized in \cite{desjardins1999homogeneous}, only a modelization is tractable and meaningful. The stochastic boundary condition appearing in \eqref{Intro equation} is interpreted as the physical law describing the driving mechanism on the atmosphere-ocean interface, i.e. as a balance of the shear stress of the ocean and the horizontal wind force, see \cite{lions1993models} for details.

\subsection{Main Results}\label{sec:main results}
We begin by introducing some notation. 
Consider a complete filtered probability space $(\Omega,\mathcal{F},(\mathcal{F}_t)_{t\geq0},\mathbf{P})$, a separable Hilbert space $\mathcal{H}$ and a cylindrical $\mathcal{F}-$Brownian motion $(W_{\HH}(t))_{t\geq 0}$ on $\mathcal{H}$. We say that a process $\Phi$ is $\mathcal{F}$-progressive measurable if $\Phi|_{(0,t)\times \Omega}$ is $\mathcal{F}_t\times \mathcal{B}((0,t))$-measurable for all $t>0$, where $\mathcal{B}$ denotes the Borel $\sigma$-algebra. For the relevant notation on function spaces, we refer to \Cref{sss:notation}.

\begin{hypothesis}\label{assumptions boundary noise}
Let 
$q>2,\ p>2q,\  \alpha\in [0,\frac{1}{q}-\frac{2}{p})$ be such that:
\begin{equation*}
 -\frac{1}{q}-\alpha+\frac{1}{2}> 0.
\end{equation*}
Assume that $h_b:(0,T)\times \Omega \to W^{-\alpha,q}(\Gamma_u;\mathcal{H})$ is a $\mathcal{F}$-progressively measurable process with $\mathbf{P}-a.s.$\ paths in 
$
L^p(0,T;W^{-\alpha,q}(\Gamma_u;\mathcal{H})) .
$
\end{hypothesis}
\begin{remark}\label{rem: example noise}
\autoref{assumptions boundary noise} is for instance satisfied if 
     $q>2$, $   p>2q>4 $ and $ \alpha=0$.
    Note that the case $q=2$ considered in \cite{da1993evolution} is not allowed in \autoref{assumptions boundary noise}. 
\end{remark}

Following the idea of \cite{da2002two} we split the analysis of \eqref{Intro equation} in two parts. First we consider the stochastic linear problem with non-homogeneous boundary conditions
\begin{equation}
\left\{
\label{Linear stochastic}
\begin{aligned}
\partial_{t}w+\nabla \rho  & =\Delta w, \qquad\quad  &\text{ on }&(0,T)\times \Dom,\\
\operatorname{div}w  & =0, &\text{ on }&(0,T)\times \Dom,\\
w    & = 0, &\text{ on }&(0,T)\times \Gamma_b,\\
\partial_{z}w_1  & =h_b\dot{ W}_{\mathcal{H}}, &\text{ on }&(0,T)\times \Gamma_u,\\
w_2   & =0, &\text{ on }&(0,T)\times \Gamma_u, \\ 
w (0) & =0, &\text{ on }&\Dom,
\end{aligned}
\right.
\end{equation}
The solution to the above linear equation \eqref{Linear stochastic} can be treated in mild form as in \cite{da1993evolution,da1996ergodicity}. Secondly, denoting by $v=u-w$ we will consider the Navier-Stokes equations with random coefficients 
\begin{equation}
\left\{
\begin{aligned}
\partial_{t}v+(v+w)\cdot\nabla(v+w)+\nabla(P-\rho)  & =\Delta v, \quad &\text{ on }&(0,T)\times \Dom,\\
\operatorname{div}v  & =0,&\text{ on }&(0,T)\times \Dom,\\
v   & = 0, &\text{ on }&(0,T)\times \Gamma_b,\\
\partial_{z}v_1   & =0, &\text{ on }&(0,T)\times \Gamma_u,\\
v_2    & =0, &\text{ on }&(0,T)\times \Gamma_u, \\ 
v(0)   & =u_0, &\text{ on }&\Dom .  
\end{aligned}
\right.
\label{modified NS}   
\end{equation}
As discussed in \cite[Chapter 13]{da1996ergodicity}, if $h_b$, $u_0$, $W_{\mathcal{H}}(t)$ would be regular enough, then $u=v+w$ will be a classical solution of the Navier-Stokes equations with inhomogeneous boundary conditions \eqref{Intro equation}. 

To state our first result, we introduce some more notation. Here and below, we denote by $H$ (resp. $V$, $\mathbb{L}^4$) the space $L^2(\Dom;\R^2)$ (resp. $H^1(\Dom;\R^2)$, $ L^4(\Dom;\R^2)$) of divergence free vector fields adapted to our framework, introduced rigorously in \Cref{subsec:stokes operator }.
\begin{definition}\label{def z weak sol}
A process $u$ with paths $\mathbf{P}-a.s.$ in $C([0,T];H)\cap L^4(0,T;\mathbb{L}^4)$ and progressively measurable with respect to these topologies, is a pathwise weak solution of \eqref{Intro equation} if $u=v+w$, where $w$ has paths in $ C(0,T;H)\cap L^4(0,T;\Ls^4(\Dom))$, it is progressively measurable with respect to these topologies and is a mild solution of \eqref{Linear stochastic} while $v$ has paths in $C(0,T;H)\cap L^2(0,T;V)$, it is progressively measurable with respect to these topologies and is a weak solution of \eqref{modified NS}. 
\end{definition}

The first main result of this paper reads as follows.

\begin{theorem}[Global well-posedness]
\label{main thm}
Let \autoref{assumptions boundary noise} be satisfied. Then for all $u_0\in H$ there exists a unique weak solution $u$  to \eqref{Intro equation} in the sense of \autoref{def z weak sol}.
\end{theorem}
According to \autoref{rem: example noise}, the introduction of the non-Hilbertian setting is necessary in order to prove \autoref{main thm} above, at least with the tools introduced in this article.

\begin{remark}[Additional bulk forces]
\label{additional bulk forces}
    Without additional difficulties we could also consider in equation \eqref{Intro equation} an additive noise diffused inside the domain of the form $h_d(t) \, \dd\widetilde{W}_{\mathcal{H}}(t)$, where $\widetilde{W}_{\mathcal{H}}$ is a cylindrical Brownian motion on $\mathcal{H}$ independent of $W_{\mathcal{H}}$ and $h_d:(0,T)\times \Omega\rightarrow \gamma(\mathcal{H}, {X_{-\lambda,A_q}} )$ is a progressively measurable process with paths $\mathbf{P}-a.s.$ in $L^p\left(0,T;\gamma\left(\mathcal{H},{X_{-\lambda,A_q}}\right)\right)$, with $p>2,\ q\geq 2,\ \lambda\in [0,\frac{1}{2}) $ such that $1-\frac{2}{p}-2\lambda>0$ and there exists $\theta \in [0,\frac{1}{2})$  satisfying
    \begin{align*}
        \theta \leq \frac{3}{4}-\lambda-\frac{1}{q},\qquad\theta\geq \frac{1}{p}-\frac{1}{4}.
    \end{align*} 
    The case $q=p=2$ and $\lambda=0$ is also allowed, see \cite[Chapter 5]{da2014stochastic}.
    Here $A_q$ and $\gamma$ stands for the Stokes operator on $L^q$ and the class of $\gamma$-radonifying operators, see \Cref{subsec:stokes operator } and  \cite[Chapter 9]{Analysis2}, respectively. {Finally, $X_{-\lambda,A_q}$ is the extrapolated space or order $\lambda$ w.r.t.\ $A_q$ as defined below, see \eqref{eq:duality_fractional_scale}.} 
    To see this, note that, under these assumptions, arguing as in \autoref{regularity stokes} the solution $\widetilde{w}$ to 
    \begin{equation}\label{Linear stochastic diffused}
    \left\{
\begin{aligned}
\partial_t \widetilde{w}+\nabla \widetilde{\rho}   & =\Delta \widetilde{w} +h_d \dot{\widetilde{W}}_{\mathcal{H}},\quad &\text{ on }&(0,T)\times \Dom,\\
\operatorname{div}\widetilde{w}  & =0, &\text{ on }&(0,T)\times \Dom, \\
\widetilde{w}    & = 0, &\text{ on }&(0,T)\times \Gamma_b,\\
\partial_{z}\widetilde{w}_1  & =0, &\text{ on }&(0,T)\times \Gamma_u,\\
\widetilde{w}_2    & =0, &\text{ on }&(0,T)\times \Gamma_u,  \\ 
\widetilde{w}(0)    & =0, &\text{ on }& \Dom,
\end{aligned}
\right.
\end{equation}
can be obtained as a stochastic convolution. In particular, the above assumptions on $h_d$ imply that $\widetilde{w}$ 
is a progressively measurable process with values in $C([0,T];H)\cap L^4(0,T;\mathbb{L}^4)$. Therefore this term adds no difficulties in order to analyze the well-posedness of equation \eqref{modified NS}. For this reason we prefer to not consider this classical source of randomness.
\end{remark}

\begin{remark}[Comparison with the literature]\
\begin{enumerate}[leftmargin=*]
\item\label{paragone berselli}
\autoref{main thm} shares strong similarities with \cite[Theorem 1.2]{berselli2006existence}, which addresses the well-posedness of certain 2D deterministic Navier-Stokes equations with non-homogeneous non-smooth Navier-type boundary conditions. However, it is important to note that our model focuses on a different phenomenon than the one studied in \cite{berselli2006existence}. For this reason, contrary to us, they stress the regularity of the boundary condition of the normal trace of the velocity. From a mathematical viewpoint, the white noise appearing in equation \eqref{Intro equation} is rougher both in time and in space compared to the boundary conditions discussed in \cite{berselli2006existence}. However, as discussed in \cite{da1993evolution}, Neumann boundary conditions are more regular than Dirichlet boundary conditions and allow us to treat rougher inputs. Due to these differences, the two results have different ranges of applicability and do not cover each other. Moreover, the tools introduced here differ significantly from the techniques involved in \cite{berselli2006existence}.
\item \label{paragone primitive}
As discussed in the introduction, the first result in the direction of the analysis of fluid dynamical models with stochastic boundary conditions have been proved in \cite[Theorem 5.1]{binz2020primitive}, where the authors established local well-posedness of 3D primitive equations with boundary noise modeling wind forces. Both their strategy and ours are based on the splitting technique introduced in \cite{da2002two}. After showing suitable regularity properties of the stochastic convolution via stochastic maximal $L^p$-regularity techniques (cf. \autoref{regularity stokes} and \cite[Proposition 4.3]{binz2020primitive}), a thorough analysis of certain nonlinear models is required. In contrast, we conduct this analysis within a suitable Hilbertian framework, enabling us to derive energy estimates essential for establishing the global well-posedness of \eqref{Intro equation} (cf. \autoref{Thm deterministic well--posed modified} and \cite[Section 5.3]{binz2020primitive}). The difference between the global well-posedness result which we are able to obtain and \cite[Theorem 5.1]{binz2020primitive} can be seen as consequence of the fact that the 2D Navier-Stokes equations are globally well-posed in the weak setting, while the same cannot be asserted for the primitive equations (cf.\ \cite{Ju17}). Therefore, in order to prove their local well-posedness result, the authors in \cite{binz2020primitive} need to work with a notion of solution which mixes strong and weak regularity in the space variables. As a byproduct of this fact we are able to consider a noise rougher in space compared to them. Additionally, a minor distinction lies in the boundary conditions applied to the bottom part of the domain $\Gamma_b$. We introduce no-slip boundary conditions to accurately model the bottom of the ocean, a choice with theoretical underpinning in works such \cite{dalibard2009asymptotic, dalibard2009mathematical, gill1982atmosphere, pedlosky1996ocean,pedlosky2013geophysical}. In contrast, \cite{binz2020primitive} considered some form of homogeneous Neumann boundary conditions, a choice related to the functional analytic setup of the primitive equations (cf. \cite[Remark 3.3]{binz2020primitive}). Beyond the distinct justifications from a modeling perspective, our choice leads to differences in the analysis of the corresponding linear elliptic systems (cf. \Cref{subsec Neumann map} and \cite[Section 3.5]{binz2020primitive}).
\end{enumerate}
\end{remark}

Secondly, we are interested in studying the interior regularity of the solution $u$ provided by \autoref{main thm}.

Our second main result reads as follows:
\begin{theorem}[Interior regularity]\label{interior regularity theorem}
Let \autoref{assumptions boundary noise} be satisfied. 
    Let $u$ be the unique weak solution of  \eqref{Intro equation} provided by \autoref{main thm}. Then for all  $t_0\in (0,T)$ and $\Dom_0\subset \Dom $ such that $\mathrm{dist}({\Dom_0}, \partial \Dom)>0$,
    \begin{align*}
        u\in C([t_0,T]; C^{\infty}({\Dom_0};\mathbb{R}^2)) \quad \mathbf{P}-a.s.
    \end{align*}
\end{theorem}
According to \cite{serrin1961interior} (see also \cite[Section 13.1]{lemarie2018navier}), it seems not possible to gain high-order interior time-regularity for the Navier-Stokes problem. This fact is in contrast to the case of the heat equation with white noise boundary conditions, see \cite{brzezniak2015second}. The reason behind this is the presence of the unknown pressure $P$ which, due to its non-local nature, provides a connection between the interior and the boundary  regularity. Finally, let us mention that other techniques to bootstrap further interior space regularity (e.g.\ analyticity), such as the `parameter' trick (see \cite{parameter1,parameter2} and \cite[Subsection 9.4]{pruss2016moving}), seem not to work due to the presence of the noise on $\Gamma_u$.
Similarly to the proof of \autoref{main thm}, we analyze the interior regularity of $u$ combining the interior regularity of $w$ and the interior regularity of $v$. The interior regularity of $w$ is obtained introducing a proper weak formulation, see \autoref{weak solution linear} below. Instead the regularity of $v$ is analyzed via a Serrin's argument exploiting the aforementioned regularity of $w$.

\smallskip

The paper is organized as follows. In \Cref{sec preliminaries} we will introduce the functional setting in order to deal with problem \eqref{Intro equation}. In particular, we will introduce the corresponding of the classical spaces and operator needed to deal with Navier-Stokes equations with no-slip boundary condition to this more involved set of boundary conditions. Indeed, the Stokes operator associated to our problem generates an analytic semigroup which admits an $H^{\infty}$-calculus of angle strictly less than $\frac{\pi}{2}$ also in the non-Hilbertian setting. This is crucial in order to apply the Stochastic maximal $L^p$-regularity results of \cite{VNeerven}, recalled in \Cref{sec: stoch maximal regularity}.
The proof of \autoref{main thm} is the object of \Cref{sec well posed}. In particular, in \Cref{sec regularity mild stokes} we will consider the linear problem \eqref{Linear stochastic}, while in \Cref{sec regularity nonlinear auxiliary} we will consider the nonlinear problem \eqref{modified NS}. The proof of \autoref{interior regularity theorem} is the object of \Cref{sec:interior reg}. In particular, in \Cref{sec interior regularity mild stokes} we will study the interior regularity of the solution of the linear problem \eqref{Linear stochastic}, while in \Cref{sec interior regularity nonlinear auxiliary} we will consider the nonlinear problem \eqref{modified NS}. We postpone some technical proofs related to the properties of the Stokes operator in the \Cref{appendinx H infty calulus}.

\subsubsection{Notation} 
\label{sss:notation}
Here we collect some notation which will be used throughout the paper. Further notation will be introduced where needed.
By $C$ we will denote several constants, perhaps changing value line by line. If we want to keep track of the dependence of $C$ from some parameter $\xi$ we will use the symbol $C(\xi)$. Sometimes we will use the notation $a \lesssim b$ (resp.\ $a\lesssim_\xi b$), if it exists a constant such that $a \leq C b$ (resp.\ $a\leq C(\xi) b$).

Fix $q\in (1,\infty)$. For an integer $k\geq 1$, $W^{k,q}$ denotes the usual Sobolev spaces. 
In the non-positive and non-integer case $s\in (-\infty,\infty)\setminus \N$, we let $W^{s,q}:=B^s_{q,q}$ where $B^{s}_{q,q}$ is the Besov space with smoothness $s$, and integrability $q$ and microscopic integrability $q$ (in particular $W^{0,q}\neq L^q$). Moreover, $H^{s,q}$ denotes the Bessel potential spaces.
Both Besov and Bessel potential spaces can be defined by means of Littlewood-Paley decompositions and restrictions (see e.g.\ \cite{ST87}, \cite[Section 6]{Saw_Besov}) or using the interpolation methods starting with the standard Sobolev spaces $W^{k,q}$ (see e.g.\ \cite[Chapter 6]{BeLo}). Finally, we set $\mathcal{A}^{s,q}(D;\R^d):=(\mathcal{A}^{s,q}(D))^d$ for an integer $d\geq 1$, a domain $D$ and $\mathcal{A}\in \{W,H\}$.

Let $\mathcal{K}$ and $Y$ be a Hilbert and a Banach space, respectively. We denote by $\gamma(\mathcal{K},Y)$ the set of $\gamma$-radonifying operators, see e.g.\ \cite[Chapter 9]{Analysis2} for basic definitions and properties. If $Y$ is Hilbert, then $\gamma(\mathcal{K},Y)$ coincides with the class of Hilbert-Schmidt operator from $\mathcal{K}$ to $Y$.
Below, we need the following Fubini-type result:
\begin{equation*}
\mathcal{A}^{s,q}(D;\mathcal{K})=\gamma(\mathcal{K},\mathcal{A}^{s,q}(D))\ \ \text{ for all }s\in\R,\ q\in (1,\infty), \ \mathcal{A}\in \{W,H\}.    
\end{equation*}
The above follows from \cite[Theorem 9.3.6]{Analysis2} and interpolation.

\section{Preliminaries}\label{sec preliminaries}
\subsection{The Stokes operator and its spectral properties}\label{subsec:stokes operator }
In this section we introduce the functional analytic setup in order to define all the object necessarily in the following. In order to improve the readability of the results we will just state the main results on the Stokes operator postponing the proofs to \Cref{appendinx H infty calulus}.

Throughout this subsection we let $q\in (1,\infty)$. Recall that $\Dom=\Tor\times (0,a)$ where $a>0$.
We begin by introducing the Helmholtz projection on $L^q(\Dom;\R^2)$, see e.g.\ \cite[Subsection 7.4]{pruss2016moving}. Let $f\in L^q(\Dom;\R^2)$ and let $\psi_f\in W^{1,q}(\Dom)$ be the unique solution to the following elliptic problem
\begin{equation}
\label{eq:psi_f_problem}
\left\{
\begin{aligned}
\Delta  \psi_f &= \div f\quad &\text{ on }&\Dom,\\
\partial_n  \psi_f &= f\cdot n   &\text{ on }&\Gamma_u \cup \Gamma_b.
\end{aligned}
\right.
\end{equation}
Here $n$ denotes the exterior normal vector field on $\partial\Dom$. Of course, the above elliptic problem is interpret in its natural weak formulation:
\begin{equation}
\int_{\Dom} \nabla \psi_f \cdot \nabla \varphi \,\dd x\dd z = \int_{\Dom} f\cdot \nabla \varphi\,\dd x \dd z\ \  \text{ for all }    \varphi\in C^{\infty}(\Dom).
\end{equation}
By \cite[Corollary 7.4.4]{pruss2016moving} , we have $\psi_f\in W^{1,q}(\Dom)$ and $\|\nabla \psi_f\|_{L^{q}(\Dom;\R^2)}\lesssim \|f\|_{L^q(\Dom;\R^2)}$ (the proof of such estimate can also be obtained by the Lax-Milgram theorem in Banach spaces \cite[Theorem 1.1]{KY13_LaxMilgram}, see also the proof of \autoref{Regularity of the Neumann map} below). Then the Helmholtz projection is given by $\p_q$ is defined as 
\begin{equation*}
\p_q f= f- \nabla \psi_f, \quad f\in L^q(\Dom;\R^2).    
\end{equation*}

Next we define the Stokes operator on $L^q(\Dom;\R^2)$. For convenience of notation, we actually define $A_q$ as minus the Stokes operator so that $A_q$ is a positive operator for $q=2$ (i.e.\ $\langle A_2 u,u \rangle\geq 0$ for all $u\in \Do(A_2)$). Let $\Ls^q:=\p(L^q(\Dom;\R^2))$. Then, we define the operator $A_q:\Do(A_q)\subseteq \Ls^q\to \Ls^q$ where
\begin{align*}
\Do(A_q)= \big\{f=(f_1,f_2)\in W^{2,q}(\Dom;\mathbb{R}^2)\cap \Ls^q\,:\,
\ &   f|_{\Gamma_b}=0, \\  
&   f_2 |_{\Gamma_u}=\partial_z f_1|_{\Gamma_u}=0\big\},
\end{align*}
and  $A_q u=-\p_q \Delta u$ for $u\in \Do(A_q)$.

In the main arguments we need stochastic maximal $L^q$-regularity estimates for stochastic convolutions. By \cite{VNeerven} (see also \cite{AV20,NVVW15_survey}), it is enough to show the boundedness of the $H^{\infty}$-calculus for $A_q$. For the main notation and basic results on the $H^{\infty}$-calculus we refer to \cite[Chapters 3 and 4]{pruss2016moving} and \cite[Chapter 10]{Analysis2}. 

In the following, we let 
\begin{equation*}
\Hs^{s,q}(\Dom):=H^{s,q}(\Dom;\R^2)\cap \Ls^q, \ \ s\in \R.    
\end{equation*}

\begin{theorem}[Boundedness $H^{\infty}$-calculus]
\label{t:bounded_H_infty}
For all $q\in (1,\infty)$, the operator $A_q$ is invertible and has a bounded $H^{\infty}$-calculus of angle $<\frac{\pi}{2}$.
Moreover the domain of the fractional powers of $A_q$ is characterized as follows:
\begin{enumerate}[leftmargin=*]
    \item $\Do(A_q^{s})= \Hs^{2s,q}(\Dom)$ if $0\leq s<\frac{1}{2q}$.
\vspace{0.1cm}
    \item  
    $\Do(A_q^{s})=
\big\{f\in \Hs^{2s,q}(\Dom)\,|\, f|_{\Gamma_b} =0,\ f_2|_{\Gamma_u}=0 \big\}
$ if $\frac{1}{2q}<s<\frac{1}{2}+\frac{1}{2q}$.
\vspace{0.1cm}
\item  
$
\Do(A_q^{s})=
\big\{f\in \Hs^{2s,q}(\Dom)\,|\, f|_{\Gamma_b} =0,\ f_2|_{\Gamma_u}=\partial_z f_1|_{\Gamma_u}=0 \big\}
$ if $\frac{1}{2}+\frac{1}{2q}<s<1$.
\end{enumerate}
\end{theorem}

The above implies that $-A_q$ generates an analytic semigroup on $\Ls^q$. \\
For convenience of notation, we will simply write $A$ in place of $A_2$. Moreover we define 
\begin{align*}
    H:=\mathbb{L}^2,\quad V:=\Do(A^{1/2}),\quad 
 \mathcal{D}(\Dom):=\{f\in C^{\infty}_{\mathrm{c}}(\Dom;\mathbb{R}^2)\, :\, \div f=0\}.
\end{align*}
We denote by $\langle\cdot,\cdot\rangle$ and $\lVert\cdot\rVert$ the inner product and the norm in $H$ respectively.
In the sequel we will denote by $V^{*}$ the dual of $V$ and we will identify $H$
with $H^{*}$. Every time $X$ is a reflexive Banach space such that the embedding $X\hookrightarrow H$ is continuous and dense, denoting by $X^*$ the dual of $X$, the scalar product $\left\langle
\cdot,\cdot\right\rangle $ in $H$ extends to the dual pairing between $X$ and $X^{*}$. We will simplify the notation accordingly.

\autoref{t:bounded_H_infty}  could be known to experts. 
For the reader's convenience,  we provide in \Cref{appendinx H infty calulus} a complete and relatively short proof based on the recent strategy used in \cite{KW17_stokes} for the $H^{\infty}$-calculus for the Stokes operator on Lipschitz domains  \cite[Theorem 16]{KW17_stokes}. 
\subsection{The Neumann map}\label{subsec Neumann map}
Now we are interested in $L^q$-estimates for the Neumann map, i.e.\ we are interested in studying the weak solutions of the elliptic problem
\begin{equation}\label{Non Homogeneous Stokes}
\left\{
\begin{aligned}
-\Delta u+\nabla \pi   &=0, & \text{ on }& \Dom,\\
\div u   &=0,\qquad & \text{ on }& \Dom,  \\
u(  \cdot,0)     &=0, & \text{ on }& \Gamma_b,\\
\partial_{z}u_1(  \cdot ,a)  & =g, & \text{ on }& \Gamma_u,\\
u_2     &=0, & \text{ on }& \Gamma_u.
\end{aligned}\right.
\end{equation}
To state the main result of this subsection, we need to formulate \eqref{Non Homogeneous Stokes} in the weak setting. To this end, we argue formally. Take $\varphi=(\varphi_1,\varphi_2)\in C^{\infty}(\Dom;\R^2)$ such that $\div \varphi=0$, 
\begin{equation*}
\varphi(\cdot,0)=0, \quad\text{ and } \quad \varphi_2( \cdot,a) =0.    
\end{equation*}
A formal integration by parts shows that \eqref{Non Homogeneous Stokes} implies
\begin{equation}
\label{Non Homogeneous Stokes weak}
    \int_{\Dom} \nabla u: \nabla \varphi \,\dd x \dd z= -\int_{\T} g(x) \varphi_1(x,a)\,\dd x.
\end{equation}

In particular, the RHS of \eqref{Non Homogeneous Stokes weak} makes sense even in case $g$ is a distribution if we interpret $\int_{\T} g(x) \varphi_1(x,a)\,\dd x =\langle \varphi_1(\cdot ,a),g\rangle$.

\begin{theorem}
    \label{Regularity of the Neumann map}
Let $q\in (1,\infty)$, for all $g\in W^{-1/q,q}(\Gamma_u)$ there exists a unique $(u,\pi)\in W^{1,q}(\Dom;\mathbb{R}^2)\times L^q(\Dom)/ \mathbb{R}$ weak solution of \eqref{Non Homogeneous Stokes}. Moreover $(u,\pi)$ satisfy 
\begin{align}
 \label{eq:elliptic_regularity_1}
     \lVert u \rVert_{W^{1,q}(\Dom;\mathbb{R}^2)}+\lVert \pi \rVert_{L^{q}(\Dom)/\mathbb R}\leq C\lVert g\rVert_{W^{-1/q,q}(\Gamma_u)}.
 \end{align}
Finally, if $g\in W^{1-1/q,q}(\Gamma_u)$, then 
 $(u,\pi)\in W^{2,q}(\Dom;\mathbb{R}^2)\times W^{1,q}(\Dom)/ \mathbb{R}$ and
 \begin{align}
 \label{eq:elliptic_regularity_2}
     \lVert u \rVert_{W^{2,q}(\Dom;\mathbb{R}^2)}+\lVert \pi \rVert_{W^{1,q}(\Dom)/\mathbb R}\leq C\lVert g\rVert_{W^{1-1/q,q}(\Gamma_u)}.
 \end{align}
\end{theorem}

\begin{proof}
We divide the proof into three steps.

\emph{Step 1: Proof of \eqref{eq:elliptic_regularity_1}}. Let $A_q$ be as in \Cref{sec preliminaries}. We prove \eqref{eq:elliptic_regularity_1} by applying the  Lax-Milgram theorem 
of \cite[Theorem 1.1]{KY13_LaxMilgram} to the form $a:Y_1\times Y_2\to \R$ where
\begin{equation*}
a(u,\varphi)= 
    \int_{\Dom} \nabla u: \nabla \varphi \,\dd x\dd z,\quad Y_1=\Do(A_q^{1/2}), \quad Y_2=\Do(A_{q'}^{1/2}).    
\end{equation*}
Recall that, by \autoref{t:bounded_H_infty},
\begin{equation*}
\Do(A_q^{1/2})=\{v=(v_1,v_2)\in \Hs^{1,q}(\Dom)\,:\, v|_{\Gamma_b}=0,\  v_2|_{\Gamma_u}=0\}.  
\end{equation*}
Since $W^{1,q'}(\Dom)\ni \varphi \mapsto \varphi_1|_{\Gamma_u}\in W^{1-1/q',q'}(\Gamma_u)= W^{1/q,q'}(\Gamma_u)$, we have  
\begin{equation}
|\langle \varphi_1(\cdot ,a),g\rangle|\leq \|g\|_{W^{-1/q,q}(\Gamma_u)}\|\varphi\|_{W^{1/q,q'}(\Gamma_u)}
\lesssim \|g\|_{W^{-1/q,q}(\Gamma_u)}\|\varphi\|_{W^{1,q'}(\Dom)}.
\end{equation}
Hence the Lax-Milgram theorem of
of \cite[Theorem 1.1]{KY13_LaxMilgram} implies the existence of $u$ as in \eqref{eq:elliptic_regularity_1} provided, for all $v\in 
\Do(A_p^{1/2})$,
\begin{equation}
\label{eq:variational_ch_Lq_norm}
\|\nabla v\|_{L^q(\Dom;\R^2)}\eqsim \sup\Big\{\int_{\Dom} \nabla v :\nabla f \,\dd x\dd z\, \Big|\,f\in \Do(A^{1/2}_{q'})\text{ and } \| f\|_{\Do(A^{1/2}_{q'})}\leq 1 \Big\}.
\end{equation}
The case $\gtrsim$ of \eqref{eq:variational_ch_Lq_norm} follows from the H\"older inequality. To prove the opposite inequality, we argue by duality. We start by discussing some known facts about the ``Sobolev tower'' of spaces associated the operator $A_p$:
\begin{align*}
X_{\alpha,A_q}&= \Do(A_p^{\alpha}) \  &\text{ for }&\alpha\geq 0,\\
X_{\alpha,A_q}&= (\Ls^q,\|A_q^{{\alpha}} \cdot\|_{\mathbb{L}^q})^{\sim} \  &\text{ for }&\alpha<0.
\end{align*}
Here $\sim$ denotes the completion (since $0\in \rho(A_q)$ by \autoref{t:bounded_H_infty}, we have that $f\mapsto \|A_q^{{\alpha}} f\|_{\Ls^q}$ is a norm {for all $\alpha<0$}). 
Since $(A_q)^*=A_{q'}$, it follows that (see e.g.\ \cite[Chapter 5, Theorem 1.4.9]{Amann95})
\begin{equation}
\label{eq:duality_fractional_scale}
(X_{\alpha,A_q})^*= X_{-\alpha,A_{q'}}.
\end{equation}
Now we can proceed in the proof of $\lesssim$ in \eqref{eq:variational_ch_Lq_norm}. Firstly, as $\Do(A_q)\ \hookrightarrow \Do(A^{1/2}_q)$ is dense for all $q\in (1,\infty)$, we can prove such inequality assuming $v\in \Do(A_q)$. In the latter case, the duality \eqref{eq:duality_fractional_scale} and the Hahn-Banach theorem imply the existence of $g\in X_{-\alpha,A_{q'}}$ of unit norm such that 
\begin{align*}
\|A_q^{1/2} v\|_{L^q(\Dom;\R^2)}
&=\int_{\Dom} A^{1/2}_q v \cdot A_{q'}^{-1/2} g  \,\dd x\dd z\\
&\stackrel{(i)}{=}\int_{\Dom} A_q v \cdot A_{q'}^{-1} g  \,\dd x \dd z\\
&\stackrel{(ii)}{=}-\int_{\Dom} \Delta v \cdot A_{q'}^{-1} g  \,\dd x \dd z\\
&\stackrel{(iii)}{=}-\int_{\Dom} \nabla v : \nabla (A^{-1}_{q'} g)  \,\dd x \dd z
\end{align*}
where in $(i)$ we used that $A_q^{1/2}v =A_{q}^{-1/2} (A_q v)$ and $(A_{q}^{-1/2})^*=A_{q'}^{-1/2}$, in $(ii)$ that $A_q=-\p_q \Delta_q$ and therefore $\p_{q'}A_{q'}^{-1} g=A_{q'}^{-1} g$ as $A_{q'}^{-1} g\in \Do(A^{1/2}_{q'})\subseteq \Ls^{q'}(\Dom)$. Finally, in $(iii)$ we used that no boundary terms appear due to the boundary conditions and $\partial_z v_1(\cdot,a)=0$ as $v\in \Do(A_q)$.

Hence the case $\lesssim$ of \eqref{eq:variational_ch_Lq_norm} follows from the above chain of equality, the fact that $\Do(A_q^{1/2})\embed W^{1,q}(\Dom;\R^2)$ and $A^{-1}_{q'}:X_{-1/2,A_{q'}}\to X_{1/2,A_{q'}}$ is an isomorphism.

Now, the existence of the pressure $\pi$ satisfying the estimate \eqref{eq:elliptic_regularity_1} is standard and follows from the De Rham theorem, see e.g.\ \cite[Corollary III.5.1, Lemma IV.1.1]{G11_book}.

\emph{Step 2: Proof of \eqref{eq:elliptic_regularity_2}}. By Step 1, it suffices to prove the existence of a solution $(u,\pi)\in W^{2,q}(\Dom)\times W^{1,q}(\Dom)/\R$ for which \eqref{Non Homogeneous Stokes} holds. In case of $g\in C^{\infty}(\Gamma_u)$, the conclusion follows from standard $L^2$-theory and we will present the argument in this case at the end of the proof. In the remaining case we argue by density. Note that, arguing as in the proof of \autoref{prop:max_reg_stokes}, a localization argument and \cite[Theorem 7.2.1]{pruss2016moving} (applied with time as a dummy variable) yield the following a-priori estimates for solutions $(u,\pi)\in W^{2,q}(\Dom;\mathbb{R}^2)\times W^{1,q}(\Dom)/\R$ to \eqref{Non Homogeneous Stokes}:
\begin{align*}
    \|u\|_{ W^{2,q}(\Dom;\mathbb{R}^2)}+ \|\nabla \pi\|_{ W^{1,q}(\Dom;\mathbb{R}^2)}\leq C\|u\|_{W^{2-2/q,q}(\Dom;\mathbb{R}^2)}+ C\|g\|_{W^{1-1/q,q}(\Gamma_u)}&\\
    \leq  \varepsilon \|u\|_{W^{2,q}(\Dom;\mathbb{R}^2)}+C_{\varepsilon}\|u\|_{W^{1,q}(\Dom;\mathbb{R}^2)} + C\|g\|_{W^{1-1/q,q}(\Gamma_u)}&\\
    \leq  \varepsilon \|u\|_{W^{2,q}(\Dom;\mathbb{R}^2)}+C_{\varepsilon}\|g\|_{W^{1-1/q,q}(\Gamma_u)}&,
\end{align*}
where $\varepsilon>0$ is arbitrary and in the last step we applied Step 1.

    The above shows $\|u\|_{ W^{2,q}(\Dom;\mathbb{R}^2)}+ \|\nabla \pi\|_{ W^{1,q}(\Dom)}\lesssim \|g\|_{W^{1-1/q,q}(\T)}$ for all solutions $(u,\pi)\in W^{2,q}(\Dom;\mathbb{R}^2)\times W^{1,q}(\Dom)/\R$ to \eqref{Non Homogeneous Stokes}. Combining this, the density of $C^{\infty}(\Gamma_u)$ in $ W^{1-1/q,q}(\Gamma_u)$, and the above mentioned solvability for $g\in C^{\infty}(\Gamma_u)$; one readily obtains the existence of solutions to \eqref{Non Homogeneous Stokes} in the class $W^{2,q}(\Dom)\times W^{1,q}(\Dom)/\R$.
    
\emph{Step 3: Proof of the regularity of $(u,\pi)$ in case of $g\in C^{\infty}(\mathbb{T})$.} The proof of this fact follows the lines of \autoref{Proposition well posed L2}. First, by Lax-Milgram Lemma and \cite[Proposition 1.1, Proposition 1.2]{temam2001navier}, there exists a unique couple, $(u,\pi)\in V\times L^2(\Dom)$ such that  
\begin{align}
    \int_{\Dom} \nabla u:\nabla \phi \,\dd x\dd z&=-\int_{\mathbb T} g(x)\phi_1(x,a)\,\dd x \quad \forall \phi \in V \label{weak formulation L2 boundary}\\
    \langle -\Delta u+\nabla\pi,\varphi\rangle&=0 \quad \forall \varphi \in C^{\infty}_c(\Dom;\mathbb{R}^2) \label{distribution formulation L2 boundary}\\
    \lVert u\rVert_V+\lVert \pi\rVert_{L^2/\mathbb{R}}&\lesssim \lVert g\rVert_{H^{-1/2}(\Gamma_u)}\label{first energy estimate L2 boundary}.
\end{align}
Now, let us fix $h>0$, extend periodically either $u$ and $g$ in the $x$ direction and consider $\phi=\tau_{h}\tau_{-h} u$ as a test function in \eqref{weak formulation L2 boundary}, where $\tau_h v=\frac{v(x+h,z)-v(x,z)}{h}.$ Then by change of variables, it follows that
\begin{align*}
    \lVert \tau_h \nabla u\rVert_{L^2(\Dom)}^2& \leq C \lVert \tau_{h} g\rVert_{L^2(\Gamma_u)} \lVert \tau_{h} u\rVert_{L^2(\Gamma_u)}\\ & \leq C \lVert \tau_{h} g\rVert_{L^2(\Gamma_u)}\lVert \tau_{h} u\rVert_{L^2(\Dom)}+C \lVert \tau_{h} g\rVert_{L^2(\Gamma_u)}\lVert \tau_{h} \nabla u\rVert_{L^2(\Dom)}\\ & \leq C\lVert g\rVert_{C^1(\Gamma_u)}\lVert \tau_h u\rVert_{L^2(\Dom)}+\frac{\lVert \tau_{h} \nabla u\rVert_{L^2(\Dom)}^2}{2}+C\lVert g\rVert_{C^1(\Gamma_u)}^2.
\end{align*}
Therefore 
\begin{align}\label{second energy estimate L^2 boundary}
  \lVert \tau_h \nabla u\rVert_{L^2(\Dom)}^2\leq  C\lVert g\rVert_{C^1(\Gamma_u)}\lVert \tau_h u\rVert_{L^2(\Dom)}+C\lVert g\rVert_{C^1(\Gamma_u)}^2. 
\end{align}
Since $u\in V$ and \eqref{first energy estimate L2 boundary} holds the right hand side of inequality \eqref{second energy estimate L^2 boundary} is uniformly bounded in $h\rightarrow 0$ and this implies \begin{align}
    \lVert \partial_x \nabla u\rVert_{L^2(O)}^2\leq C\lVert g\rVert_{C^1(\Gamma_u)}^2.
\end{align}
Let us now consider $\phi=\partial_x\psi,\ \psi\in \mathcal{D}(\Dom)$ as test function in \eqref{weak formulation L2 boundary}. Thanks to \cite[Proposition 1.1, Proposition 1.2]{temam2001navier}, $\partial_x \pi\in L^2(\Dom)$ and $\lVert \partial_x\pi\rVert_{L^2}\lesssim \lVert g\rVert_{C^1(\Gamma_u)} $. Since $u$ is divergence free and \eqref{distribution formulation L2 boundary} holds, then \begin{align*}
    \partial_z \pi=\partial_{xx}u_2-\partial_{xz}u_1\in L^2(\Dom).
\end{align*}
Therefore $\lVert \nabla\pi\rVert_{L^2}\lesssim \lVert g\rVert_{C^1(\Gamma_u)}.$
Lastly, again by relation \eqref{distribution formulation L2 boundary}
\begin{align*}
    \partial_{zz}u_1=\partial_x\pi-\partial_{xz}u_1\in L^2(\Dom)
\end{align*}
Combining all the information obtained we get 
\begin{align*}
    \lVert u\rVert_{H^{2}(\Dom;\mathbb{R}^2)}+\lVert \pi\rVert_{H^1(\Dom)/\mathbb{R}}\leq C\lVert g\rVert_{C^1(\Gamma_u)}^2
\end{align*}
Iterating the argument one gets that $(u,\pi)\in H^{k+1}(\Dom;\mathbb{R}^2)\times H^{k}(\Dom)$ provided $g\in C^k(\Gamma_u)$ for all $k\geq 1$. Now the claim of Step 3 follows from Sobolev embeddings.
\end{proof}

Next we denote by $\NM$ the solution map defined by \autoref{Regularity of the Neumann map} which associate to a boundary datum $g$ the velocity $u$ solution of \eqref{Non Homogeneous Stokes}, i.e.\ $\NM g:=u$. 
From the above result we obtain 

\begin{corollary}\label{Corollary regularity Neumann map}
Let $\NM$ and $\mathcal{H}$ be the Neumann map and a Hilbert space, respectively. 
Then, for all $q\geq 2$ and $\varepsilon>0$,
\begin{enumerate}[leftmargin=*]
    \item\label{it:mapping_NM1} $\NM\in \mathscr{L}(W^{-\alpha,q}(\Gamma_u;\mathcal{H});\gamma(\mathcal{H},\Do(A_q^{\frac{1-\alpha}{2}+\frac{1}{2q}-\epsilon})))$ for $\alpha\in [0,\frac{1}{q}]$.
    \item\label{it:mapping_NM2} $\NM\in \mathscr{L}(L^q({\Gamma_u};\mathcal{H});\gamma(\mathcal{H},\Do(A_q^{\frac{1}{2}+\frac{1}{2q}-\epsilon})))$.
\end{enumerate}
\end{corollary}

\begin{proof}
To begin, recall that $W^{s,q}(\Gamma_u;\mathcal{H})=\gamma(\mathcal{H},W^{s,q}(\Gamma_u))$ for all $s\in \R$ and $ q\in (1,\infty)$, see \Cref{sss:notation}. Hence, due to the ideal property of $\gamma$-radonifying operators \cite[Theorems 9.1.10 and 9.1.20]{Analysis2}, it is enough to consider the scalar case $\mathcal{H}=\R$.

\eqref{it:mapping_NM1}: 
By interpolating with the real method $(\cdot,\cdot)_{\theta,q}$ where $\theta\in (0,1)$ (see e.g.\ \cite[Theorem 6.4.5]{BeLo}), the estimates in \autoref{Regularity of the Neumann map} yield 
\begin{equation*}
\NM : W^{\theta-1/q,q} (\Gamma_u)\to W^{\theta+1,q}(\Dom)\ \  \text{ for all }\theta\in (0,1).    
\end{equation*}
Moreover, by construction $\NM [u] $ satisfies 
\begin{equation*}
\NM [u]|_{\Gamma_b}=0 , \quad \text{ and }\quad (\NM [u])_2 |_{\Gamma_u}=0,    
\end{equation*}
where $(\cdot)_2$ denotes the second component.
Hence \eqref{it:mapping_NM1} follows from the description of the fractional power of $A_q$ in \autoref{t:bounded_H_infty}
and that $B^{1+\theta}_{q,q}(\Dom;\R^2)\embed H^{\theta+1-\varepsilon,q}(\Dom;\R^2)$.

\eqref{it:mapping_NM2}: Follows from \eqref{it:mapping_NM1} and $L^q(\Gamma_u)\embed B^{0}_{q,q}(\Gamma_u)$ as $q\geq 2$.
\end{proof}

\subsection{Deterministic Navier-Stokes equations}
Let us consider the deterministic Navier-Stokes equations with homogeneous boundary conditions 

\begin{equation}
\left\{
\begin{aligned}
\partial_{t}\overline{u}+\overline{u}\cdot\nabla\overline{u}+\nabla\overline{\pi}  & =\Delta \overline{u}+\overline{f}, \quad &\text{ on }&(0,T)\times \Dom,\\
\operatorname{div}\overline{u}  & =0,&\text{ on }&(0,T)\times \Dom,\\
\overline{u}   & = 0, &\text{ on }&(0,T)\times \Gamma_b,\\
\partial_{z}\overline{u}_1   & =0, &\text{ on }&(0,T)\times \Gamma_u,\\
\overline{u}_2    & =0, &\text{ on }&(0,T)\times \Gamma_u, \\ 
\overline{u}(0)   & =\overline{u}_0, &\text{ on }&\Dom .  
\end{aligned}
\right.
\label{classical NS}   
\end{equation}

Define the trilinear form $b:\mathbb{L}^{4}\times V\times\mathbb{L}%
^{4}\rightarrow\mathbb{R}$ as%
\begin{equation}
\label{def:b_trilinear_form}
b\left(  u,v,w\right)  =\sum_{i,j=1}^{2}\int_{\Dom}u_{i}
\partial_{i}v_{j}  w_{j}  \, \dd x\dd z=\int_{\Dom}\left(
u\cdot\nabla v\right)  \cdot w\,\dd x\dd z
\end{equation}
which is well--defined and continuous on $\mathbb{L}^{4}\times V\times
\mathbb{L}^{4}$ by the H\"{o}lder inequality. Since by Sobolev embedding theorem $V\subset\mathbb{L}^{4}$, $b$ is also defined and continuous on
$V\times V\times V$. Moreover, by standard interpolation inequalities,
\begin{equation}
\lVert f\rVert _{L^{4}\left(  \Dom\right)  }^{2}\leq C\lVert
f\rVert _{L^{2}\left(  \Dom\right)  }\lVert f\rVert _{H^{1}\left(
\Dom\right)  } \ \ \text{ for all }f\in H^{1}\left(  \Dom\right)  . \label{interpolation ineq}%
\end{equation}
 Integrating by parts, the standard oddity relation below holds
\[
b\left(  u,v,w\right)  =-b\left(  u,w,v\right)
\]
if $u\in\mathbb{L}^{4}$, $v,w\in V$.

Lastly we introduce the operator
\[
B:\mathbb{L}^{4}\times\mathbb{L}^{4}\rightarrow V^*%
\]
defined by the identity%
\[
\left\langle B\left(  u,v\right)  ,\phi\right\rangle =-b\left(  u,\phi
,v\right)  =-\int_{\Dom}\left(  u\cdot\nabla\phi\right)  \cdot v\,\dd x\dd z
\]
for all $\phi\in V$. When $v\in V$, we may also write%
\[
\left\langle B\left(  u,v\right)  ,\phi\right\rangle =b(  u,v,\phi
).\]
Moreover, when $u\cdot\nabla v\in L^{2}\left(  \Dom;\mathbb{R}^{2}\right)  $, it
is explicitly given by
\[
B\left(  u,v\right)  =\p(  u\cdot\nabla v).
\]
We have to define our notion of weak solution for problem \eqref{classical NS}.
\begin{definition}
\label{Def NS determ}Given $\overline{u}_{0}\in H$ and $\overline{f}\in L^{2}\left(  0,T;V^{*}\right)  $, we say that
\[
\overline{u}\in C\left(  \left[  0,T\right]  ;H\right)  \cap L^{2}\left(  0,T;V\right)
\]
is a weak solution of equation \eqref{classical NS} if 
for all $\phi\in \Do\left(  A\right)  $ and $t\in [0,T]$,
\begin{align*}
&  \left\langle \overline{u}\left(  t\right)  ,\phi\right\rangle -\int_{0}^{t}b\left(
\overline{u}\left(  s\right)  ,\phi,\overline{u}\left(  s\right)  \right)  \,\dd s\\
&  =\left\langle \overline{u}_{0},\phi\right\rangle -\int_{0}^{t}\left\langle \overline{u}\left(
s\right)  ,A\phi\right\rangle \,\dd s+\int_{0}^{t}\left\langle \overline{f}\left(  s\right)
,\phi\right\rangle_{V^*,V} \,\dd s.
\end{align*}
\end{definition}

The following results
are simple adaptations of classical results, see for instance  \cite{flandoli2023stochastic,lions1996mathematical,temam1995navier,temam2001navier}.
\begin{lemma}
\label{lemma tecnico}If $u,v\in L^{4}\left(  0,T;\mathbb{L}^{4}\right)  $
then
\begin{equation}
B\left(  u,v\right)  \in L^{2}\left(  0,T;V^*\right)  .
\label{ineq 1 lemma tecnico}%
\end{equation}
Moreover,%
\begin{equation}
\lvert b\left(  u,v,w\right)  \rvert \leq\epsilon\lVert
v\rVert _{V}^{2}+\epsilon'\lVert u\rVert _{V}^{2}+\frac{C%
}{\epsilon^2\epsilon'}\lVert u\rVert^{2}\lVert w\rVert
_{\mathbb{L}^4}^{4} \label{ineq 2 lemma tecnico}%
\end{equation}%
\begin{equation}
\lvert b\left(  u,v,w\right)  \rvert \leq\epsilon\lVert
v\rVert _{V}^{2}+\epsilon'\lVert w\rVert _{V}^{2}+\frac{C%
}{\epsilon^2\epsilon'}\lVert w\rVert^{2}\lVert u\rVert
_{\mathbb{L}^4}^{4}, \label{ineq 3 lemma tecnico}%
\end{equation}
where $C$ is a constant independent of $\epsilon$ and $\epsilon'$.
\end{lemma}

\begin{theorem}
\label{Thm deterministic well--posed}For every $\overline{u}_{0}\in H$ and $\overline{f}\in
L^{2}\left(  0,T;V^*\right)  $ there exists a unique weak solution of
equation \eqref{classical NS}. It satisfies%
\[
\lVert \overline{u}\left(  t\right)  \rVert^{2}+2\nu\int_{0}%
^{t}\lVert \nabla \overline{u}\left(  s\right)  \rVert _{L^{2}}^{2}\,\dd s=\lVert
\overline{u}_{0}\rVert^{2}+2\int_{0}^{t}\left\langle \overline{u}\left(  s\right)
,\overline{f}\left(  s\right)  \right\rangle_{V^*,V} \,\dd s.
\]
If $\left(  \overline{u}_{0}^{n}\right)  _{n\in\mathbb{N}}$ is a sequence in $H$
converging to $\overline{u}_{0}\in H$ and $\left(  \overline{f}^{n}\right)  _{n\in\mathbb{N}}$ is a
sequence in $L^{2}\left(  0,T;V^*\right)  $ converging to $\overline{f}\in
L^{2}\left(  0,T;V^*\right)  $, then the corresponding unique solutions
$\left(  \overline{u}^{n}\right)  _{n\in\mathbb{N}}$ converge to the corresponding
solution $\overline{u}$ in $C\left(  \left[  0,T\right]  ;H\right)  $ and in
$L^{2}\left(  0,T;V\right)  $.
\end{theorem}

\subsection{Stochastic Maximal $L^p$-regularity}\label{sec: stoch maximal regularity}
Let $\mathcal{H}$ and $(W_{\mathcal{H}}(t))_{t\geq 0}$ be a Hilbert space and a cylindrical $\mathcal{F}-$Brownian motion on $\mathcal{H}$, respectively.
The following result was proven in \cite{VNeerven}, see also \cite[Section 7]{NVVW15_survey} and \cite[Section 3]{AV22_QSEE1} for additional references. 
Below, for a Banach space $Y$, $H^{s,q}(\R_+;Y)$ denotes the $Y$-valued Bessel potential space on $\R_+$ with smoothness $s\in \R$ and integrability $q$; such space can be defined either by complex interpolation (see e.g.\ \cite[Chapter 3, Section 4.5]{pruss2016moving}) or by restriction from $\R$ (see e.g.\ \cite[Subsection 3.1]{ALV23}).
For the notion of $H^{\infty}$-calculus and $\gamma$-radonifying operators $\gamma(\mathcal{H},Y)$ we refer to \cite[Chapter 9 and 10]{Analysis2}.

\begin{theorem}\label{max regularity}
    
    Let $X$ be a Banach space isomorphic to a closed subspace of $L^q(D,\mu)$ where $q\in [2,+\infty)$ and $(D,\mathscr{A},\mu)$ is a $\sigma$-finite measure space. Let $\mathcal{A}$ be an invertible operator and assume that it admits a bounded $H^{\infty}$ calculus of angle $<\pi/2$ on $X$ and let $(\mathcal{S}(t))_{t\geq 0}$ the bounded analytic semigroup generated by $-\mathcal{A}$. For all $\mathcal{F}-$adapted $G\in L^p(\mathbb{R}_+\times \Omega;\gamma(\mathcal{H};X))$ the stochastic convolution process
    \begin{align*}
        U(t)=\int_0^t \mathcal{S}(t-s)G(s)\,\dd W_{\mathcal{H}}(s)\quad t\geq 0,
    \end{align*}
    is well defined in $X$, takes values in the fractional domain $\Do(\mathcal{A}^{1/2})$ almost surely and for all $2<p<+\infty$ the following space-time regularity estimate holds: $\forall \theta\in [0,\frac{1}{2})$
    \begin{align}\label{maximal regularity inequality}
    \mathbf{E}\left[\lVert U(t) \rVert^p_{H^{\theta,p}(\mathbb{R}_+;\Do(\mathcal{A}^{1/2-\theta}))} \right]\leq C_{\theta}^p\mathbf{E}\left[\lVert G\rVert_{L^p(\mathbb{R}_+;L^q(O;\mathcal{H}))}^p\right]        
    \end{align}
    with a constant $C_{\theta}$ independent of $G$.
\end{theorem}

For extensions of the above result we refer to \cite{AV20,LV21} for the weighted case, and to \cite[Subsection 6.2]{ALV23} for the case of homogeneous spaces.
However, the latter situations will not be considered here.

\section{Well-Posedness}\label{sec well posed}
\subsection{Stokes Equations}\label{sec regularity mild stokes}
As discussed in \Cref{sec:main results}, we start by considering the linear problem \eqref{Linear stochastic}.
According to \cite{da1993evolution,da1996ergodicity}, the mild solution $w$ of the former problem is formally given by
\begin{align}\label{Mild equation linear}
    w(t)=A_q\int_0^t S_q(t-s)\mathcal{N}[h_b(s)]\,\dd W_{\mathcal{H}}(s).
\end{align}
Here $A_q$ is (minus) the Stokes operator with homogeneous boundary conditions, and $(S_q(t))_{t\geq 0}$ its corresponding semigroup (cf.\ \autoref{t:bounded_H_infty}).

Next step is to prove that $w(t)$ is well defined in some functional spaces and has some regularities useful to treat the nonlinearity of the Navier-Stokes equations.

\begin{proposition}\label{regularity stokes}
    Let $\alpha\in [0,\frac{1}{q}]$ and assume that $h_b:(0,T)\times \Omega\to W^{-\alpha,q}(\Gamma_u;\mathcal{H})$ is $\mathcal{F}$-progressive measurable with $\mathbf{P}-a.s.$ paths in  $L^p(0,T;W^{-\alpha,q}(\Gamma_u))$. Then the process $w$ defined in \eqref{Mild equation linear} is a well defined process with  $\mathbf{P}-a.s.$ paths in 
    \begin{equation*}
     H^{\theta,p}(0,T;\Do(A_q^{\frac{1}{2q}-\frac{\alpha}{2}-\theta-\epsilon})) \ \ \text{ for all }\  \theta\in [0,\tfrac{1}{2}),\ \epsilon>0.   
    \end{equation*}
In particular, if $h_b$ satisfies \autoref{assumptions boundary noise}, then $w$ has $\mathbf{P}-a.s.$ trajectories in $ C([0,T];H)\cap L^4(0,T;\mathbb{L}^4)$.
\end{proposition}

 \begin{proof}
     By replacing $h_b$ by $\one_{[0,\tau_n]\times \Omega}h_{b}$, $\tau_n$ being the following stopping time 
     \begin{equation*}
    \tau_n:=\{t\in [0,T]\,:\, \|h\|_{L^p(0,t;W^{-\alpha,q}(\Gamma_u;\mathcal{H}))}\geq n\} \quad \text{ where }\quad \inf\varnothing :=T,
     \end{equation*}
     for all $n\geq 1$, it is enough to consider the case $h_b\in L^p((0,T)\times \Omega;W^{-\alpha,q}(\Gamma_u;\mathcal{H}))$.
     
     Let $\epsilon>0$ be fixed later. From \autoref{Corollary regularity Neumann map} and \autoref{max regularity} we have that $ \mathbf{P}-a.s.$ and  for each $\theta\in [0,\frac{1}{2})$
     \begin{align*}
         \widetilde{w}=\int_0^\cdot S_q(\cdot-s)A_q^{\frac{1-\alpha}{2}+\frac{1}{2q}-\epsilon}\mathcal{N}[h_b(s) ]\,\dd W_{\mathcal{H}}(s)\in H^{\theta,p}(0,T;\Do(A_q^{1/2-\theta})) \quad
     \end{align*}
     Therefore, $\mathbf{P}-a.s.$,  
 \begin{equation*}
     w=A_q^{\frac{1+\alpha}{2}-\frac{1}{2q}+\epsilon}\widetilde{w}\in  H^{\theta,p}(0,T;\Do(A_q^{\frac{1}{2q}-\frac{\alpha}{2}-\theta-\epsilon})).
 \end{equation*}
Finally, note that, by \autoref{assumptions boundary noise}, \autoref{t:bounded_H_infty} and the Sobolev embeddings (see e.g.\ \cite[Proposition 2.7]{AV22_QSEE1}) we can find $\theta_1,\theta_2\in [0,\frac{1}{2})$ and $\varepsilon>0$ such that 
\begin{align*}
          H^{\theta_1,p}(0,T;\Do(A_q^{\frac{1}{2q}-\frac{\alpha}{2}-\theta_1-\epsilon}))&\hookrightarrow C([0,T];H),\\ 
          H^{\theta_2,p}(0,T;\Do(A_q^{\frac{1}{2q}-\frac{\alpha}{2}-\theta_2-\epsilon}))&\hookrightarrow L^4(0,T;\mathbb{L}^4).
     \end{align*}
     where the first embedding follows from $\alpha<\frac{1}{q}-\frac{2}{p}$ and the second one from the remaining conditions in \autoref{assumptions boundary noise}.
     Hence the proof is complete.
     \end{proof}

\subsection{Auxiliary Navier--Stokes Type Equations}\label{sec regularity nonlinear auxiliary}

Having solved the Stokes problem we
introduce the auxiliary variable%
\[
v\left(  t\right)  =u\left(  t\right)  -w(  t),
\]
which satisfies \eqref{modified NS}, i.e.
\begin{equation*}
\left\{
\begin{aligned}
\partial_{t}v+\left(  v+w\right)  \cdot\nabla\left(  v+w\right)
+\nabla\left(  P-\rho\right)   &  =\nu\Delta v, &\text{ on }&(0,T)\times \Dom,\\
\operatorname{div}v  & =0,&\text{ on }&(0,T)\times \Dom,\\
v   & = 0, &\text{ on }&(0,T)\times \Gamma_b,\\
\partial_{z}v_1   & =0, &\text{ on }&(0,T)\times \Gamma_u,\\
v_2    & =0, &\text{ on }&(0,T)\times \Gamma_u, \\ 
v(0)   & =u_0, &\text{ on }&\Dom .  
\end{aligned}
\right.
\end{equation*}
This first equation in the above system has the form%
\begin{equation*}
\partial_{t}v+v\cdot\nabla v+\nabla\pi  =\nu\Delta v-L\left(  v,w\right) 
\end{equation*}
with the affine function%
\[
L\left(  v,w\right)  =v\cdot\nabla w+w\cdot\nabla v+w\cdot\nabla w.
\]
For each $\omega\in \Omega$ fixed, the Navier--Stokes structure is preserved and the auxiliary
equation for $v$ with homogeneous boundary conditions is solvable similarly to the classical Navier--Stokes
equations. 
Therefore, let us introduce the notion of weak solution of the deterministic problem \eqref{modified NS} with random coefficients.
Recall that $A$ and $b$ are (minus) the Stokes operator on $\Ls^2$ and defined in \eqref{def:b_trilinear_form}, respectively.

\begin{definition}
\label{Def NS determ modified}Given $u_0\in H$ and $w\in L^{4}\left(  0,T;\mathbb{L}^{4}\right)  $,
we say that
\[
v\in C\left(  \left[  0,T\right]  ;H\right)  \cap L^{2}\left(  0,T;V\right)
\]
is a weak solution of equation \eqref{modified NS} if
\begin{align*}
&  \left\langle v\left(  t\right)  ,\phi\right\rangle -\int_{0}^{t}b\left(
v\left(  s\right)  +w\left(  s\right)  ,\phi,v\left(  s\right)  +w\left(
s\right)  \right)  \,\dd s\\
&  =\left\langle u_0,\phi\right\rangle -\int_{0}^{t}\left\langle v\left(
s\right)  ,A\phi\right\rangle \,\dd s
\end{align*}
for every $\phi\in \Do\left(  A\right)  $ and $t\in [0,T]$.
\end{definition}
\begin{theorem}
\label{Thm deterministic well--posed modified}For every $u_0\in H$ and $w\in L^{4}\left(  0,T;\mathbb{L}%
^{4}\right)  $, there exists a unique weak solution $v$ of equation
\eqref{modified NS}. Moreover, $v$ satisfies for all $t\in [0,T]$%
\begin{align}\label{energy equality modified}
&  \lVert v\left(  t\right)  \rVert^{2}+2\int_{0}%
^{t}\lVert \nabla v\left(  s\right)  \rVert _{L^{2}}^{2}\,\dd s=\lVert
u_0\rVert^{2} +2\int_{0}^{t}\left(  b\left(  v,v,w\right) 
+b\left(  w,v,w\right)  \right)  \left(  s\right)  \,\dd s.
\end{align}
If $\left(  u_{0}^{n}\right)  _{n\in\mathbb{N}}$ is a sequence in $H$
converging to $u_{0}\in H$ and $\left(  w^{n}\right)  _{n\in\mathbb{N}}$ is a
sequence in $L^{4}\left(  0,T;\mathbb{L}%
^{4}\right) $ converging to $w\in
L^{4}\left(  0,T;\mathbb{L}%
^{4}\right)  $, then the corresponding unique solutions
$\left(  v^{n}\right)  _{n\in\mathbb{N}}$ converge to the corresponding
solution $v$ in $C\left(  \left[  0,T\right]  ;H\right)  $ and in
$L^{2}\left(  0,T;V\right)  $.
\end{theorem}

\begin{proof}
We split the proof into several steps.

\emph{Step 1: Uniqueness}. Let $v^{\left(  i\right)  }$ be two solutions.
The function $z=v^{\left(  1\right)  }-v^{\left(  2\right)  }$ satisfies
\begin{align*}
&  \left\langle z\left(  t\right)  ,\phi\right\rangle -\int_{0}^{t}\left(
b\left(  v^{\left(  1\right)  }+w,\phi,v^{\left(  1\right)  }+w\right)
-b\left(  v^{\left(  2\right)  }+w,\phi,v^{\left(  2\right)  }+w\right)
\right)  \,\dd s\\
&  =-\int_{0}^{t}\left\langle z\left(  s\right)  ,A\phi\right\rangle \,\dd s
\end{align*}
for every $\phi\in \Do\left(  A\right)  $. A simple manipulation gives us%
\begin{align*}
&  b\left(  v^{\left(  1\right)  }+w,\phi,v^{\left(  1\right)  }+w\right)
-b\left(  v^{\left(  2\right)  }+w,\phi,v^{\left(  2\right)  }+w\right)
-b\left(  z,\phi,z\right) \\
&  =b\left(  v^{\left(  2\right)  }+w,\phi,z\right)  +b\left(  z,\phi
,v^{\left(  2\right)  }+w\right)
\end{align*}
hence%
\begin{align*}
&  \left\langle z\left(  t\right)  ,\phi\right\rangle -\int_{0}^{t}b\left(
z\left(  s\right)  ,\phi,z\left(  s\right)  \right)  \,\dd s=-\int_{0}^{t}\left\langle z\left(  s\right)  ,A\phi\right\rangle
\,\dd s+\int_{0}^{t}\left\langle \widetilde{f}\left(  s\right)  ,\phi\right\rangle
\,\dd s
\end{align*}
where
\[
\widetilde{f}=-B\left(  v^{\left(  2\right)  }+w,z\right)  -B\left(
z,v^{\left(  2\right)  }+w\right)\hspace{-0.1cm}.
\]
By \autoref{lemma tecnico}, $\widetilde{f}\in L^{2}\left(
0,T;V^*\right)  $. Then, by \autoref{Thm deterministic well--posed}%
,
\[
\lVert z\left(  t\right)  \rVert^{2}+2\int_{0}%
^{t}\lVert \nabla z\left(  s\right)  \rVert _{L^{2}}^{2}\,\dd s=2\int%
_{0}^{t}b\left(
z,z,v^{\left(  2\right)  }+w\right)\left(s\right)    \,\dd s.
\]
Again by \autoref{lemma tecnico}, we have%
\begin{align*}
\left\lvert b\left(z,z,  v^{\left(  2\right)  }+w\right)  \right\rvert  &
\leq\left\lvert b\left( z,z, v^{\left(  2\right)  }\right)  \right\rvert
+\left\lvert b\left(z,z,w\right)  \right\rvert \\
&  \leq\epsilon\lVert z\rVert _{V}^{2}+\epsilon\lVert
z\rVert _{V}^{2}+\frac{C}{\epsilon^{3}}\lVert z\rVert
^{2}\lVert v^{\left(  2\right)  }\rVert _{\mathbb{L}^4}^{4}\\
&  +\epsilon\lVert z\rVert _{V}^{2}+\epsilon\lVert z\rVert
_{V}^{2}+\frac{C}{\epsilon^{3}}\lVert z\rVert^{2}\lVert
w\rVert _{\mathbb{L}^4}^{4}%
\end{align*}%
\[
=4\epsilon\lVert z\rVert _{V}^{2}+\frac{C}{\epsilon^{3}%
}\lVert z\rVert^{2}\left(  \lVert v^{\left(  2\right)
}\rVert _{\mathbb{L}^4}^{4}+\lVert w\rVert _{\mathbb{L}^4}^{4}\right)\hspace{-0.1cm}.
\]
Summarizing, with $4\epsilon=1$, using the fact that $\lVert
z\rVert _{V}^{2}=\lVert \nabla z\rVert _{L^{2}}^{2}$, renaming the constant $C$,
\[
\lVert z\left(  t\right)  \rVert^{2}+\int_{0}%
^{t}\lVert \nabla z\left(  s\right)  \rVert _{L^{2}}^{2}\,\dd s{\leq} C\int%
_{0}^{t}\lVert z\left(  s\right)  \rVert^{2}\left(  1+\lVert
v^{\left(  2\right)  }\left(  s\right)  \rVert _{\mathbb{L}^4}^{4}+\lVert
w\left(  s\right)  \rVert _{\mathbb{L}^4}^{4}\right)  \,\dd s.
\]
We conclude $z=0$ by the Gronwall lemma, using the assumption on $w$ and
the integrability properties of $v^{\left(  2\right)  }$.

\emph{Step 2: Existence}. 
Define the sequence $\left(  v^{n}\right)  $ by
setting $v^{0}=0$ and for every $n\geq0$, given $v^{n}\in C\left(  \left[
0,T\right]  ;H\right)  \cap L^{2}\left(  0,T;V\right)  $, let $v^{n+1}$ be the
solution of equation \eqref{classical NS} with initial condition $u_0$ and
with
\[
-B\left(  v^{n},w\right)  -B\left(  w,v^{n}\right)  -B\left(  w,w\right)
\]
in place of $f$. In particular%
\begin{align*}
&  \left\langle v^{n+1}\left(  t\right)  ,\phi\right\rangle -\int_{0}%
^{t}b\left(  v^{n+1}\left(  s\right)  ,\phi,v^{n+1}\left(  s\right)  \right)
\,\dd s\\
&  =\left\langle u_0,\phi\right\rangle -\int_{0}^{t}\left\langle
v^{n+1}\left(  s\right)  ,A\phi\right\rangle \,\dd s\\
&  -\int_{0}^{t}\left\langle \left(  B\left(  v^{n},w\right)  +B\left(
w,v^{n}\right)  +B\left(  w,w\right)  \right)  \left(  s\right)
,\phi\right\rangle \,\dd s
\end{align*}
for every $\phi\in \Do\left(  A\right)  $. In order to claim that this
definition is well done, we notice that
\[
B\left(  v^{n},w\right)  ,B\left(  w,v^{n}\right)  ,B\left(  w,w\right)  \in
L^{2}\left(  0,T;V^*\right)
\]
by \autoref{lemma tecnico}.

Then let us investigate the convergence of $\left(  v^{n}\right)  $. First,
let us prove a bound. From the previous identity and
\autoref{Thm deterministic well--posed} we get
\begin{align*}
&  \lVert v^{n+1}\left(  t\right)  \rVert^{2}+2\int%
_{0}^{t}\lVert \nabla v^{n+1}\left(  s\right)  \rVert _{L^{2}}%
^{2}\,\dd s\\
&  =\lVert u_0\rVert^{2}  +2\int_{0}^{t}\left(  b\left(  v^{n},v^{n+1},w\right)  +b\left(
w,v^{n+1},v^{n}\right)  +b\left(  w,v^{n+1},w\right)  \right)  \left(
s\right)  \,\dd s.
\end{align*}
It gives us (using \autoref{lemma tecnico})
\begin{align*}
&  \lVert v^{n+1}\left(  t\right)  \rVert^{2}+\int%
_{0}^{t}\lVert \nabla v^{n+1}\left(  s\right)  \rVert _{L^{2}}%
^{2}\,\dd s\\
&  =\lVert u_0\rVert^{2}+\epsilon\int_{0}^{t}\lVert
v^{n}\left(  s\right)  \rVert _{V}^{2}\,\dd s\\
&  +C_{\epsilon}\int_{0}^{t}\lVert v^{n}\left(  s\right)  \rVert^{2}\left(  1+\lVert w\left(  s\right)  \rVert _{\mathbb{L}^4}%
^{4}\right)  \,\dd s+C_{\epsilon}\int_{0}^{t}\lVert w\left(  s\right)
\rVert _{\mathbb{L}^4}^{4}\,\dd s.
\end{align*}
Choosing a small constant $\epsilon$, one can find
$R>\lVert u_0\rVert^{2}$ and $\overline{T}$ small enough, depending only from $\lVert u_0\rVert$ and $\lVert w\rVert_{L^4(0,T;\mathbb{L}^4)}$, such that
if
\begin{equation}
\sup_{t\in\left[  0,\overline{T}\right]  }\lVert v^{n}\left(  t\right)  \rVert
^{2}\leq R,\qquad\int_{0}^{\overline{T}}\lVert v^{n}\left(  s\right)  \rVert
_{V}^{2}\,\dd s\leq R \label{fixed point R}%
\end{equation}
then the same inequalities hold for $v^{n+1}$.

Set $w_{n}=v^{n}-v^{n-1}$, for $n\geq1$. From the identity above,%
\begin{align*}
&  \left\langle w_{n+1}\left(  t\right)  ,\phi\right\rangle -\int_{0}%
^{t}\left(  b\left(  v^{n+1},\phi,v^{n+1}\right)  -b\left(  v^{n},\phi
,v^{n}\right)  \right)  \left(  s\right)  \,\dd s\\
&  =-\int_{0}^{t}\left\langle w_{n+1}\left(  s\right)  ,A\phi\right\rangle
\,\dd s-\int_{0}^{t}\left\langle \left(  B\left(  v^{n},w\right)  -B\left(
v^{n-1},w\right)  \right)  \left(  s\right)  ,\phi\right\rangle \,\dd s\\
&  -\int_{0}^{t}\left\langle \left(  B\left(  w,v^{n}\right)  -B\left(
w,v^{n-1}\right)  \right)  \left(  s\right)  ,\phi\right\rangle \,\dd s.
\end{align*}
Again as above, since%
\begin{align*}
&  b\left(  v^{n+1},\phi,v^{n+1}\right)  -b\left(  v^{n},\phi,v^{n}\right)
-b\left(  w_{n+1},\phi,w_{n+1}\right) \\
&  =b\left(  v^{n},\phi,w_{n+1}\right)  +b\left(  w_{n+1},\phi,v^{n}\right)
\end{align*}
we may rewrite it as%
\begin{align*}
&  \left\langle w_{n+1}\left(  t\right)  ,\phi\right\rangle -\int_{0}%
^{t}b\left(  w_{n+1}\left(  s\right)  ,\phi,w_{n+1}\left(  s\right)  \right)
\,\dd s\\
&  =-\int_{0}^{t}\left\langle w_{n+1}\left(  s\right)  ,A\phi\right\rangle
\,\dd s-\int_{0}^{t}\left\langle \left(  B\left(  w_{n},w\right)  +B\left(
w,w_{n}\right)  \right)  \left(  s\right)  ,\phi\right\rangle \,\dd s\\
&  +\int_{0}^{t}\left(  b\left(  v^{n},\phi,w_{n+1}\right)  +b\left(
w_{n+1},\phi,v^{n}\right)  \right)  \left(  s\right)  \,\dd s.
\end{align*}
One can check as above the applicability of
\autoref{Thm deterministic well--posed} and get
\begin{align*}
&  \lVert w_{n+1}\left(  t\right)  \rVert ^{2}+2\int%
_{0}^{t}\lVert \nabla w_{n+1}\left(  s\right)  \rVert _{L^{2}}%
^{2}\,\dd s\\
&  =2\int_{0}^{t}\left(  b\left(  w_{n},w_{n+1},w\right)  +b\left(
w,w_{n+1},w_{n}\right)  \right)  \left(  s\right)  \,\dd s\\
&  +2\int_{0}^{t}b\left(
w_{n+1},w_{n+1},v^{n}\right)\left(  s\right)  \,\dd s.
\end{align*}
As above we deduce%
\[
\lvert b\left(  w_{n+1}%
,w_{n+1},v^{n}\right)  \rvert \leq\frac{1}{2}\lVert w_{n+1}%
\rVert _{V}^{2}+C\lVert w_{n+1}\rVert^{2}\lVert
v^{n}\rVert _{\mathbb{L}^4}^{4}.
\]
But%
\[
\lvert b\left(  w_{n},w_{n+1},w\right)  +b\left(  w,w_{n+1},w_{n}\right)
\rvert \leq\frac{1}{2}\lVert w_{n+1}\rVert _{V}^{2}+\frac{1
}{4}\lVert w_{n}\rVert _{V}^{2}+C\lVert w_{n}\rVert%
^{2}\lVert w\rVert _{\mathbb{L}^4}^{4}.
\]
Hence%
\begin{align*}
&  \lVert w_{n+1}\left(  t\right)  \rVert ^{2}+\int%
_{0}^{t}\lVert \nabla w_{n+1}\left(  s\right)  \rVert _{L^{2}}%
^{2}\,\dd s\\
&  \leq C\int_{0}^{t}\lVert w_{n+1}\left(  s\right)  \rVert%
^{2}\left(  1+\lVert v^{n}\left(  s\right)  \rVert _{\mathbb{L}^4}%
^{4}\right)  \,\dd s\\
&  +\frac{1}{4}\int_{0}^{t}\lVert w_{n}\left(  s\right)  \rVert _{V}%
^{2}\,\dd s+C\int_{0}^{t}\lVert w_{n}\left(  s\right)  \rVert%
^{2}\lVert w\left(  s\right)  \rVert _{\mathbb{L}^4}^{4}\,\dd s.
\end{align*}
Now we work under the bounds \eqref{fixed point R} and deduce, using the Gronwall
lemma, for $\overline{T}$, depending only from $\lVert u_0\rVert$ and $\lVert w\rVert_{L^4(0,T;\mathbb{L}^4)}$, possibly smaller than the previous one,%
\begin{align*}
&  \sup_{t\in\left[  0,\overline{T}\right]  }\lVert w_{n+1}\left(  t\right)
\rVert^{2}+\int_{0}^{\overline{T}}\lVert w_{n+1}\left(  s\right)
\rVert _{V}^{2}\,\dd s\\
&  \leq\frac{1}{2}\left(  \sup_{t\in\left[  0,\overline{T}\right]  }\lVert
w_{n}\left(  t\right)  \rVert^{2}+\int_{0}^{\overline{T}}\lVert
w_{n}\left(  s\right)  \rVert _{V}^{2}\,\dd s\right) \hspace{-0.1cm}.
\end{align*}
It implies that the sequence $\left(  v^{n}\right)  $ is Cauchy in $C\left(
\left[  0,\overline{T}\right]  ;H\right)  \cap L^{2}\left(  0,\overline{T};V\right)  $. The limit
$v$ has the right regularity to be a weak solution and satisfies the weak
formulation; in the identity above for $v^{n+1}$ and $v^{n}$ we may prove
that%
\begin{align*}
\int_{0}^{t}b\left(  v^{n+1}\left(  s\right)  ,\phi,v^{n+1}\left(  s\right)
\right)  \,\dd s  &  \rightarrow\int_{0}^{t}b\left(  v\left(  s\right)
,\phi,v\left(  s\right)  \right)  \,\dd s\\
\int_{0}^{t}b\left(  v^{n}\left(  s\right)  ,\phi,w\left(  s\right)  \right)
\,\dd s  &  \rightarrow\int_{0}^{t}b\left(  v\left(  s\right)  ,\phi,w\left(
s\right)  \right)  \,\dd s\\
\int_{0}^{t}b\left(  w\left(  s\right)  ,\phi,v^{n}\left(  s\right)  \right)
\,\dd s  &  \rightarrow\int_{0}^{t}b\left(  w\left(  s\right)  ,\phi,v\left(
s\right)  \right)  \,\dd s.
\end{align*}
All these convergences can be proved easily by recalling the definition of $b$.
Similarly, we can pass to the limit in the energy identity. After proving existence and uniqueness in $[0,\overline{T}]$ we can reiterate the existence procedure and in a finite number of steps cover the interval $[0,T]$.

\emph{Step 3: Continuous dependence on the data}. Let $v^n$  (resp. $v$) the unique solution of \eqref{modified NS} with data $u_0^n,\ w^n$ (resp. $u_0,\ w$). Since $u_0^n\rightarrow u_0$ in $H$  (resp. $w^n\rightarrow w$ in $L^4(0,T;\mathbb{L}^4)$) the family $(u_0^n)_{n\in\mathbb{N}}$ is bounded in $H$ (resp. the family $(w^n)_{n\in\mathbb{N}}$ is bounded in $L^4(0,T;\mathbb{L}^4)$), by \eqref{energy equality modified} one can show easily that the family $(v^n)_{n\in\mathbb{N}}$ is bounded in $C([0,T];H)\cap L^2(0,T;V)\hookrightarrow L^4(0,T;\mathbb{L}^4).$ Moreover for each $t\in [0,T]$, $w^n:=v^n-v$ satisfies the energy relation
\begin{align}\label{continuity step 1}
    \frac{1}{2}\lVert w^n(t)\rVert^2+\int_0^t \lVert \nabla w^n(s)\rVert_{L^2}^2 \,\dd s&= \frac{1}{2}\lVert u_0^n-u_0\rVert^2 \notag\\ &+\int_0^t b(v^n(s)+z^n(s),w^n(s),z^n(s)-z(s))\,\dd s \notag\\ &+\int_0^t b(w^n(s),w^n(s),v(s)+z(s))\,\dd s\notag\\ &+\int_0^t b(z^n(s)-z(s),w^n(s),v(s)+z(s))\,\dd s.
\end{align}
We can easily bound the right hand side of relation \eqref{continuity step 1} by Young's inequality and H\"older's inequality obtaining 
\begin{align}\label{continuity step 2}
    \frac{1}{2}\lVert w^n(t)\rVert^2&+\frac{1}{2}\int_0^t \lVert \nabla w^n(s)\rVert_{L^2}^2 \,\dd s\leq \frac{1}{2}\lVert u_0^n-u_0\rVert^2\notag\\ &+C\int_0^t \lVert w^n(s)\rVert^2 \left(\lVert v(s)\rVert_{\mathbb{L}^4}^4+\lVert z(s)\rVert_{\mathbb{L}^4}^4\right)\,\dd s\notag\\ & +C\lVert z^n-z\rVert_{L^4(0,T;\mathbb{L}^4)}^2\left(\lVert z^n\rVert_{L^4(0,T;\mathbb{L}^4)}^2+\lVert z\rVert_{L^4(0,T;\mathbb{L}^4)}^2\right)\notag\\ &+C\lVert z^n-z\rVert_{L^4(0,T;\mathbb{L}^4)}^2\left(\lVert v^n\rVert_{L^4(0,T;\mathbb{L}^4)}^2+\lVert v\rVert_{L^4(0,T;\mathbb{L}^4)}^2\right).
\end{align}
Applying Gronwall's inequality to relation \eqref{continuity step 2} the claim follows immediately.
\end{proof}

\begin{remark}\label{remark measurability}
Freezing the variable $\omega\in \Omega$ and solving \eqref{modified NS} for each $\omega$ does not allow to obtain information about the measurability properties of $v$. However, measurability of $v$ with respect of the progressive $\sigma$-algebra follows from the continuity of the solution map with respect to $u_0$ and $w$. Therefore we have the required measurability properties for $v$ with $w$ being the mild solution of \eqref{Linear stochastic}. In particular $v$ has $\mathbf{P}$-a.s. paths in $C(0,T;H)\cap L^2(0,T;V)$, it is progressively measurable with respect to these topologies and \begin{align}\label{auxiliary weak formulation nonlinear}
&  \left\langle v\left(  t\right)  ,\phi\right\rangle -\int_{0}^{t}b\left(
v\left(  s\right)  +w\left(  s\right)  ,\phi,v\left(  s\right)  +w\left(
s\right)  \right)  \,\dd s \notag\\
&  =\left\langle u_0,\phi\right\rangle -\int_{0}^{t}\left\langle v\left(
s\right)  ,A\phi\right\rangle \,\dd s \quad \mathbf{P}-a.s.
\end{align}
for every $\phi\in \Do\left(  A\right)  $ and $t\in [0,T]$.
\end{remark}

\begin{proof}[Proof of \autoref{main thm}] It follows immediately combining \autoref{regularity stokes}, \autoref{Thm deterministic well--posed modified} and \autoref{remark measurability}.
\end{proof}

\section{Interior Regularity}\label{sec:interior reg}
\subsection{Stokes Equations}\label{sec interior regularity mild stokes}

 As in the proof of \autoref{main thm}, by a stopping time argument we may assume that $h_b$ is also $L^p(\Omega)$-integrable, cf.\ the beginning of the proof of \autoref{regularity stokes}.
This fact will be used below without further mentioning it. 
We start showing a lemma, concerning the relation between the mild and the weak formulation of \eqref{Linear stochastic} defined below.
\begin{definition}\label{weak solution linear}
Let \autoref{assumptions boundary noise} be satisfied.
A stochastic process $w$ 
is a weak solution of \eqref{Linear stochastic} if it is $\mathcal{F}$-progressively measurable with $\mathbf{P}-a.s.$ paths in
\begin{align*}
    w \in C([0,T];H)\cap L^4(0,T;\mathbb{L}^4)
\end{align*}
and for each $\phi\in \Do(A)$
\begin{align}\label{weak formulation linear}
\langle w(t),\phi\rangle=& -\int_{0}^{t}\langle w(s),A\phi\rangle
\ \,\dd s+\int_0^t \langle h_b(s),\phi\rangle_{W^{-\alpha,q}(\Gamma_u),W^{\alpha,q'}(\Gamma_u)} \,\dd W_{\mathcal{H}}(s)
\end{align}
for each $t \in[0,T],\ \mathbf{P}-a.s.$
\end{definition}

Note that the last term in \eqref{weak formulation linear} is well-defined as $\alpha<1/2$ and $q'< 2$.

\begin{remark}\label{remark well-posed stoch integral}
In the definition above, the term  
$\langle h_b(s),\phi\rangle_{W^{-\alpha,q}(\Gamma_u),W^{\alpha,q'}(\Gamma_u)}$ is given by the following linear operator on $\mathcal{H}$:
\begin{equation*}
\mathcal{H}\ni h'\mapsto \langle h_b(s)h',\phi\rangle_{W^{-\alpha,q}(\Gamma_u),W^{\alpha,q'}(\Gamma_u)}=L_{\phi} (h_b(s)h')
\end{equation*}
where $L_{\phi}:=\langle \cdot,\phi\rangle_{W^{-\alpha,q}(\Gamma_u),W^{\alpha,q'}(\Gamma_u)}$. By the ideal property of $\gamma$-spaces and 
$\gamma(\mathcal{H}, W^{-\alpha,q}(\Gamma_u))=W^{-\alpha,q}(\Gamma_u;\mathcal{H})$ we have
\begin{equation*}
\|\langle h_b(s),\phi\rangle_{W^{-\alpha,q}(\Gamma_u),W^{\alpha,q'}(\Gamma_u)}\|_{\mathcal{H}^*}\lesssim \|h_b(s)\|_{W^{-\alpha,q}(\Gamma_u;\mathcal{H})}\|\phi\|_{W^{\alpha,q'}(\Gamma_u)},
\end{equation*}
a.e.\ on $(0,T)\times \Omega$.
Whence, the stochastic integral in \eqref{weak formulation linear} is well-defined with scalar value as
\begin{align*}
    &\mathbf{E}\left[\Big|  \int_0^T \langle h_b(s),\phi\rangle_{W^{-\alpha,q}(\Gamma_u),W^{\alpha,q'}(\Gamma_u)} \,\dd W_{\mathcal{H}}(s)\Big|^2\right]\\
    &\quad= \mathbf{E} \left[ \int_0^T \|\langle h_b(s),\phi\rangle_{W^{-\alpha,q}(\Gamma_u),W^{\alpha,q'}(\Gamma_u)}\|_{\mathcal{H}^*}^2 \,\dd s\right]\\
    &\quad  \lesssim\mathbf{E} \left[ \int_0^T \|h_b(s)\|_{W^{-\alpha,q}(\Gamma_u;\mathcal{H})}^2\|\phi\|_{W^{-\alpha,q'}(\Gamma_u)}^2 \,\dd s\right]<\infty,
\end{align*}    
where the last estimate follows from \autoref{assumptions boundary noise}.
\end{remark}

\begin{lemma}\label{lemma weak solution linear equivalence}
    Let \autoref{assumptions boundary noise} be satisfied.
   There exists a unique weak solution of \eqref{Linear stochastic} in the sense of \autoref{weak solution linear} and it is given by the formula \eqref{Mild equation linear}.
\end{lemma}
\begin{proof}
We split the proof into two steps.

\emph{Step 1: There exists a unique weak solution of \eqref{Linear stochastic}  and it is necessarily given by the mild formula \eqref{Mild equation linear}}. Let $\psi \in C^1([0,T];\Do(A))$. Arguing as in the first step of the proof of \cite[Theorem 1.7]{flandoli2023stochastic}, see also \cite[Proposition 17]{flandoli2022heat}, one can readily check that $w$ satisfies 
\begin{align}\label{weak formulation linear time depending test}
\langle w(t),\psi(t)\rangle=& \int_0^t\langle w(s),\partial_s\psi(s)\rangle \,\dd s-\int_{0}^{t}\langle w(s),A\psi(s)\rangle \,\dd s\notag \\ & +\int_0^t \langle h_b(s),\psi(s)\rangle_{W^{-\alpha,q}(\Gamma_u),W^{\alpha,q'}(\Gamma_u)} \,\dd W_{\mathcal{H}}(s)
\end{align}
for each $t \in[0,T],\ \mathbf{P}-a.s.$
The stochastic integral in the relation above is well-defined arguing as in \autoref{remark well-posed stoch integral}.
Now consider $\phi\in \Do(A^2)$ and use $\psi_t(s)=S_{q'}(t-s)\phi,\ s\in [0,t]$ as test function in \eqref{weak formulation linear time depending test}  obtaining, since $S_{q'}(t)|_{H}=S(t)$, 
\begin{align}\label{weak formulation linear time depending test 2}
\langle w(t),\phi\rangle= & \int_0^t \langle h_b(s),S_{q'}(t-s)\phi\rangle_{W^{-\alpha,q}(\Gamma_u),W^{\alpha,q'}(\Gamma_u)} \,\dd W_{\mathcal{H}}(s). 
\end{align}
Recalling the definition of the Neumann map $\mathcal{N}$, \eqref{weak formulation linear time depending test 2} can be rewritten as 
\begin{align}\label{weak formulation linear time depending test 3}
\langle w(t),\phi\rangle= & \int_0^t \langle \mathcal{N}[h_b(s)],A_{q'}S_{q'}(t-s)\phi\rangle \,\dd W_{\mathcal{H}}(s).
\end{align}

Then, exploiting the self-adjointness property of $S_q$ and $A_q$ we have that weak solutions of \eqref{Linear stochastic} satisfy the mild formulation. Therefore they are unique.

\emph{Step 2: The mild formula \eqref{Mild equation linear} is a weak solution of \eqref{Linear stochastic} in the sense of \autoref{weak solution linear}}. We begin by noticing that $w$ has the required regularity due to \autoref{regularity stokes}. Let us test our mild formulation \eqref{Mild equation linear} against functions $\phi\in \Do(A^2)\subseteq \Do(A_{q'}^2)$. It holds, since $S_{q'}(t)|_{H}=S(t),\ A_{q'}|_{\Do(A)}=A$ and exploiting self-adjointness property of $S_q$ and $A_q$  \begin{align*}
\langle w(t),\phi\rangle &= \int_0^t \langle \mathcal{N}[h_b(s)],AS(t-s)\phi\rangle \,\dd W_{\mathcal{H}}(s)\\ & =\int_0^t \langle h_b(s),S(t-s)\phi\rangle_{W^{-\alpha,q}(\Gamma_u),W^{\alpha,q'}(\Gamma_u)} \,\dd W_{\mathcal{H}}(s)\quad \mathbf{P}-a.s.,
\end{align*}
where in the last step we used the definition of Neumann map. In order to complete the proof of this step it is enough to show that 
\begin{align}\label{relation to be verified}
    \int_0^t \langle h_b(s),S(t-s)\phi\rangle_{W^{-\alpha,q}(\Gamma_u),W^{\alpha,q'}(\Gamma_u)} \,\dd W_{\mathcal{H}}(s)=-\int_0^t \langle w(s),A\phi\rangle \,\dd s\notag&\\ 
    +\int_0^t \langle h_b(s),\phi\rangle_{W^{-\alpha,q}(\Gamma_u),W^{\alpha,q'}(\Gamma_u)} \,\dd W_{\mathcal{H}}(s) \quad \mathbf{P}-a.s. &
\end{align}
Relation \eqref{relation to be verified} is true. Indeed,
\begin{align}\label{weak formulation pre fubini}
    \int_0^t \langle w(s),A\phi\rangle \,\dd s&=\int_0^t \,\dd s \int_0^s \langle \mathcal{N}[h_b(r)],S(s-r)A^2 \phi \rangle \,\dd W_{\mathcal{H}}(r)\quad \mathbf{P}-a.s.
\end{align}
The double integrals in \eqref{weak formulation pre fubini} can be exchanged via stochastic Fubini's Theorem, see \cite{da2014stochastic}, since 
\begin{align*}
    &\int_0^t \,\dd s\left(\mathbf{E}\left[\int_0^s \,\dd r \| \langle  \mathcal{N}[h_b(r)],S(s-r)A^2 \phi\rangle\|_{\mathcal{H}}^2\right]\right)^{1/2}\\ & \leq C(T,q)\lVert A^2 \phi\rVert_{\mathbb{L}^2}\mathbf{E}\left[\lVert h_b\rVert_{L^2(0,T;W^{-\alpha,q}(\Gamma_u;\mathcal{H}))}^2\right]^{1/2} <+\infty\end{align*}
Therefore the double integral in the right hand side of \eqref{weak formulation pre fubini} can be rewritten as 
\begin{align*}
    &\int_0^t \,\dd s \int_0^s \langle \mathcal{N}[h_b(r)],S(s-r)A^2 \phi \rangle \,\dd W_{\mathcal{H}}(r)\\
    &\qquad =\int_0^t \,\dd W_{\mathcal{H}}(r) \int_r^t  \langle \mathcal{N}[h_b(r)],S(s-r)A^2 \phi \rangle \,\dd s\\ 
    &\qquad =\int_0^t \langle \mathcal{N}[h_b(r)],A \phi \rangle \,\dd W_{\mathcal{H}}(r)\\ 
    & \qquad  -\int_0^t \langle \mathcal{N}[h_b(r)],AS(t-r)\phi \rangle \,\dd W_{\mathcal{H}}(r)\\ 
    &\qquad  = \int_0^t \langle h_b(r), \phi \rangle_{W^{-\alpha,q}(\Gamma_u),W^{\alpha,q'}(\Gamma_u)} \,\dd W_{\mathcal{H}}(r)\\ 
    & \qquad  -\int_0^t \langle h_b(r),S(t-r)\phi \rangle_{W^{-\alpha,q}(\Gamma_u),W^{\alpha,q'}(\Gamma_u)} \,\dd W_{\mathcal{H}}(r)\quad \mathbf{P}-a.s.
\end{align*}
Inserting this expression in \eqref{weak formulation pre fubini}, \eqref{relation to be verified} holds and the proof is complete.
\end{proof}

Thanks to the weak formulation guaranteed by \autoref{lemma weak solution linear equivalence} we can easily obtain the interior regularity of the linear stochastic problem \eqref{Linear stochastic}. Let $N_0$ be the $\mathbf{P}$ null measure set where at least one between $w\notin C([0,T];H)\cap L^4(0,T;\mathbb{L}^4)$, $v\notin C([0,T];H)\cap L^2(0,T;V)$, \eqref{weak formulation linear} and \eqref{auxiliary weak formulation nonlinear} is not satisfied. In the following we will work pathwise in $\Omega\setminus N_0$ even if not specified.
\begin{corollary}\label{corollary inteorior regularity linear}
        Let \autoref{assumptions boundary noise} be satisfied.
        Let $w$ be the unique weak solution of  \eqref{Linear stochastic} in the sense of \autoref{weak solution linear}.  Then, for all $0<t_1\leq t_2<T,$ $x_0\in \Dom $, $r>0$ such that $\mathrm{dist}({B(x_0,r)}, \partial \Dom)>0$,
    \begin{align*}
        w\in C([t_1,t_2], C^{\infty}({B(x_0,r)};\mathbb{R}^2)) \quad \mathbf{P}-a.s.
    \end{align*}
\end{corollary}

\begin{proof}
    Denote $\omega_w=\operatorname{curl}w\in C([0,T];H^{-1}(\Dom))\ \mathbf{P}-a.s.$ Since $\mathrm{dist}({B(x_0,r)}, \partial \Dom)>0,\ 0<t_1\leq t_2<T$ we can find $\epsilon$ small enough such that $0<t_1-2\epsilon<t_2+2\epsilon<T,\ \mathrm{dist}({B(x_0,r+2\epsilon)}, \partial \Dom)>0$.
    Let us consider $\psi\in C^{\infty}_c(\Dom)$ and use $\nabla^{\perp}\psi$ as test function in \eqref{weak formulation linear}. This implies that $\omega_w$ is a distributional solution of the heat equation \begin{align*}
        \partial_t \omega_w=\Delta \omega_w.
    \end{align*} 
Since $\omega_w$ solves the heat equation in distributions, a standard localization argument and regularity results for the heat equation (see e.g.\ \cite[Chapter 6, Section 1]{TaylorPDE3}) imply that
\begin{align*}
    \omega_w\in C([t_1-\epsilon,t_2+\epsilon],C^{\infty}(B(x_0,r+\epsilon)))\quad \mathbf{P}-a.s.
\end{align*}
Let us now consider a test function $\phi\in C^{\infty}_c(B(x_0,r+\epsilon))$ identically equal to one on $B(x_0,r+\epsilon/2)$. Since $\operatorname{div} w=0$, we have that $\hat{w}=\phi w$ solves the elliptic problem \begin{align*}
    \Delta\hat{w}=\nabla^{\perp}\omega_w\phi +\Delta \phi w+2\nabla\phi\cdot\nabla w,\quad \hat{w}|_{\partial B(x_0,r+\epsilon)}=0.
\end{align*}
Since $w\in C([t_1-\epsilon,t_2+\epsilon];L^2(B(x_0,r+\epsilon)))\ \mathbf{P}-a.s.$ by \autoref{regularity stokes}, it follows that \begin{align*}
    \nabla^{\perp}\omega_w\phi +\Delta \phi w+2\nabla\phi\cdot\nabla w\in C([t_0-\epsilon,T];H^{-1}(B(x_0,r+\epsilon)))\quad \mathbf{P}-a.s.
\end{align*}
Therefore, by standard elliptic regularity theory \begin{align*}
    \hat{w}\in C([t_1-\epsilon,t_2+\epsilon];H^{1}(B(x_0,r+\epsilon)))\quad \mathbf{P}-a.s.
\end{align*} From the fact that $\phi\equiv 1$ on $B(x_0,r+\epsilon/2)$ it follows that 
\begin{align*}
    w\in C([t_1-\epsilon/4,t_2+\epsilon/4];H^{1}(B(x_0,r+\epsilon/4)))\quad \mathbf{P}-a.s.
\end{align*}
Therefore, the required regularity of $w$ is established by inductively reiterating this argument and by considering test functions $\phi\in C^{\infty}_c(B(x_0,r+\frac{\epsilon}{2^{2n}}))$ identically equal to one on $B(x_0,r+\frac{\epsilon}{2^{2n+1}})$. 
\end{proof}

\subsection{Auxiliary Navier--Stokes Type Equations}\label{sec interior regularity nonlinear auxiliary}
In order to deal with the interior regularity of \eqref{modified NS} we perform a Serrin type argument, see \cite{lemarie2018navier,serrin1961interior}. The regularity of $w$ guaranteed by \autoref{corollary inteorior regularity linear} will play a crucial role to treat the linear terms appearing in \eqref{modified NS}. We start with the following lemma.
\begin{lemma}\label{Lemma uniform boundendness}
    Let \autoref{assumptions boundary noise} be satisfied.
    Let $v$ be the unique solution of  \eqref{modified NS} in the sense of \autoref{Def NS determ modified}, where $w$ is as in \autoref{corollary inteorior regularity linear}. Then, for all $0<t_1\leq t_2<T,$ $x_0\in \Dom $, $r>0$ such that $\mathrm{dist}({B(x_0,r)}, \partial \Dom)>0$,
    \begin{align*}
        v\in C([t_1,t_2], H^{3/2}({B(x_0,r)};\mathbb{R}^2)) \quad \mathbf{P}-a.s.
    \end{align*}
\end{lemma}

\begin{proof}
As described in \autoref{lemma weak solution linear equivalence}, arguing as in the proof of \cite[Theorem 7]{flandoli2023stochastic}, we can extend the weak formulation satisfied by $v$ to time dependent test functions $\phi\in C^1([0,T]; H)\cap C([0,T];\Do(A))$ obtaining that for each $t\in [0,T]$
\begin{align*}
    \langle v(t),\phi(t)\rangle-\langle u_0,\phi(0)\rangle& =  \int_0^t \langle v(s),\partial_s \phi(s)\rangle \,\dd s- \int_{0}^{t}\left\langle v\left(
s\right)  ,A\phi(s)\right\rangle \,\dd s\\ & +\int_{0}^{t}b\left(
v\left(  s\right)  +w\left(  s\right)  ,\phi(s),v\left(  s\right)  +w\left(
s\right)  \right)  \,\dd s \quad \mathbf{P}-a.s.
\end{align*}
Choosing $\phi=-\nabla^{\perp}\chi,\ \chi\in C^{\infty}_c((0,T)\times \Dom)$ in the weak formulation above and denoting by 
\begin{align*}
    \omega=\operatorname{curl}v&\in C([0,T];H^{-1})\cap L^2((0,T)\times \Dom)\quad \mathbf{P}-a.s., \\ 
    \omega_w=\operatorname{curl}w&\in  C([t_1,t_2], C^{\infty}({B(x_0,r)}))\quad \mathbf{P}-a.s.
\end{align*} it follows that \begin{align*}
    -\int_0^t \langle \omega(s),\partial_s\chi(s)\rangle+\langle \omega(s),\Delta \chi(s)\rangle \,\dd s &=\int_0^t \langle \operatorname{curl}(w(s)\otimes w(s)),\nabla \chi(s)\rangle \,\dd s \\ & +\int_0^t \langle \operatorname{curl}(w(s)\otimes v(s)),\nabla \chi(s)\rangle \,\dd s\\ &  +\int_0^t\langle \operatorname{curl}(v(s)\otimes w(s)),\nabla \chi(s)\rangle  \,\dd s\\ &+\int_0^t \langle \omega(s), v(s)\cdot\nabla \chi(s)\rangle \,\dd s.
\end{align*} 
This means that $\omega$ is a distributional solution in $(0,T)\times \Dom$ of the partial differential equation
\begin{align*}
    \partial_t \omega+v\cdot\nabla \omega
    &=\Delta \omega-\operatorname{div}\bigg(\operatorname{curl}(w(s)\otimes w(s))\\
    &+\operatorname{curl}(w(s)\otimes v(s))+\operatorname{curl}(v(s)\otimes w(s))\bigg).
\end{align*}
Since $\mathrm{dist}({B(x_0,r)}, \partial \Dom)>0,\ 0<t_1\leq t_2<T$ we can find $\epsilon$ small enough such that $0<t_1-2\epsilon<t_1\leq t_2<t_2+2\epsilon<T,\ \mathrm{dist}({B(x_0,r+2\epsilon)}, \partial \Dom)>0$. Let us consider $\psi\in C^{\infty}_c((0,T)\times \Dom)$ supported in $[t_1-\epsilon,t_2+\epsilon]\times B(x_0,r+\epsilon)$ such that it is equal to one in $[t_1-\epsilon/2,t_2+\epsilon/2]\times B(x_0,r+\epsilon/2)$. Let us denote by $\Tilde{\omega}=\omega \psi\in L^2((0,T)\times \mathbb{R}^2)$ supported in $[t_1-\epsilon,t_2+\epsilon]\times B(x_0,r+\epsilon)$, then $\Tilde{\omega}$ is a distributional solution in $(0,T)\times \mathbb{R}^2$ of
\begin{align}\label{distributional solution tilde omega 1}
    \partial_t \Tilde{\omega}&= \Delta \Tilde{\omega}-v\cdot \nabla\Tilde{\omega}-w\cdot\nabla\Tilde{\omega}+g
\end{align}
with \begin{align*}
 g & =\partial_t\psi \omega-2\nabla\psi\cdot\nabla \omega-\Delta \psi \omega+v\cdot\nabla\psi \omega-\psi w\cdot\nabla\omega_w-\psi v\cdot\nabla \omega_w+w\cdot\nabla\psi \omega.
\end{align*}
Due to \autoref{corollary inteorior regularity linear}
the terms \begin{align*}
    -\psi w\cdot\nabla\omega_w-\psi v\cdot\nabla \omega_w+w\cdot\nabla\psi \omega \in L^2((0,T)\times \mathbb{R}^2)\quad \mathbf{P}-a.s.
\end{align*}
Therefore $g\in L^2(0,T;H^{-1}(\mathbb{R}^2))+L^1(0,T;L^2(\mathbb{R}^2))\ \mathbf{P}-a.s.$ Then, arguing as in the first step of the proof of \cite[Theorem 13.2]{lemarie2018navier}, the fact that $\Tilde{\omega}$ is a distributional solution of \eqref{distributional solution tilde omega 1} implies that $\Tilde{\omega}\in C([0,T];L^2(\mathbb{R}^2))\cap L^2(0,T;H^1(\mathbb{R}^2)).$
Therefore \begin{align*}
  \omega\in & C([t_1-\epsilon/4,t_2+\epsilon/4],L^2(B(x_0,r+\epsilon/4)))\\ &\cap L^2(t_1-\epsilon/4,t_2+\epsilon/4,H^1(B(x_0,r+\epsilon/4)))\quad \mathbf{P}-a.s.  
\end{align*} Introducing $\phi\in C^{\infty}_c(B(x_0,r+\epsilon/4))$ equal to one in $B(x_0,r+\epsilon/8)$, since $\omega=\operatorname{curl }v$, then $\phi v$ satisfies
\begin{align}\label{elliptic u}
    \Delta(\phi v)=\nabla^{\perp}\omega\phi+\Delta \phi v+2\nabla\phi\cdot\nabla v ,\quad (\phi v)|_{\partial  B(x_0,r+\epsilon/4)}=0.  
\end{align}
From the regularity of $\omega$, by standard elliptic regularity theory (see for example \cite{ambrosio2019lectures}), it follows that $\phi v\in C([t_1-\epsilon/4,t_2+\epsilon/4];H^1(B(x_0,r+\epsilon/4);\mathbb{R}^2))\cap L^2(t_1-\epsilon/4,t_2+\epsilon/4;H^2(B(x_0,r+\epsilon/4);\mathbb{R}^2))\ \mathbf{P}-a.s$. Therefore, since $\phi\equiv 1$ on $B(x_0,r+\epsilon/8)$ \begin{align}\label{preliminary interior regularity v}
    v\in  & C([t_1-\frac{\epsilon}{16},t_2+\frac{\epsilon}{16}];H^1(B(x_0,r+\frac{\epsilon}{16});\mathbb{R}^2))\notag \\ &\cap L^2(t_1-\frac{\epsilon}{16},t_2+\frac{\epsilon}{16};H^2(B(x_0,r+\frac{\epsilon}{16});\mathbb{R}^2))\quad \mathbf{P}-a.s.
\end{align}
Let us now consider $\hat{\psi}\in C^{\infty}_c((t_1-\frac{\epsilon}{16},t_2+\frac{\epsilon}{16})\times B(x_0,r+\frac{\epsilon}{16}))$ such that it is equal to one in $[t_1-\frac{\epsilon}{32},t_2+\frac{\epsilon}{32}]\times B(x_0,r+\frac{\epsilon}{32})$. Let us denote by $\hat{\omega}=\omega \hat{\psi}\in C([0,T];L^2(\mathbb{R}^2))\cap L^2(0,T; H^1(\mathbb{R}^2))$ supported in $(t_1-\frac{\epsilon}{16},t_2+\frac{\epsilon}{16})\times B(x_0,r+\frac{\epsilon}{16})$, then $\hat{\omega}$ is a distributional solution in $(0,T)\times \mathbb{R}^2$ of
\begin{align}\label{distributional solution tilde omega 2}
    \partial_t \hat{\omega}&= \Delta \hat{\omega}+\hat{g}
\end{align}
with \begin{align*}
 \hat{g} & =-v\cdot \nabla\hat{\omega}-w\cdot\nabla\hat{\omega}+\partial_t\hat{\psi} \omega-2\nabla\hat{\psi}\cdot\nabla \omega-\Delta \hat{\psi} \omega+v\cdot\nabla\hat{\psi} \omega\\ &-\hat{\psi} w\cdot\nabla\omega_w-\hat{\psi} w\cdot\nabla\omega_w-\hat{\psi} v\cdot\nabla \omega_w+w\cdot\nabla\hat{\psi}\omega.
\end{align*}
By \autoref{corollary inteorior regularity linear} and relation \eqref{preliminary interior regularity v} it follows that $\hat{g}\in L^2(0,T;H^{-1/2}(\mathbb{R}^2))\ \mathbf{P}-a.s.$  Therefore $\hat{\omega}\in  C([0,T];H^{1/2}(\mathbb{R}^2))\cap L^2(0,T;H^{3/2}(\mathbb{R}^2)) \ \mathbf{P}-a.s.$ and arguing as above 
\begin{align*}
    v\in  & C([t_1-\frac{\epsilon}{64},{t}_2+\frac{\epsilon}{64}], H^{3/2}({B(x_0,{r}+\frac{\epsilon}{64})};\mathbb{R}^2))\\ &\cap L^2({t}_1-\frac{\epsilon}{64},{t}_2+\frac{\epsilon}{64},  H^{5/2}({B(x_0,{r}+\frac{\epsilon}{64})};\mathbb{R}^2))\quad \mathbf{P}-a.s.
    \end{align*}
    This concludes the proof of \autoref{Lemma uniform boundendness}.
\end{proof}
\begin{corollary}\label{corollary inteorior regularity nonlinear}
        Let \autoref{assumptions boundary noise} be satisfied.
        Let $v$ be the unique weak solution of  \eqref{modified NS} in the sense of \autoref{Def NS determ modified}, where $w$ is as in \autoref{corollary inteorior regularity linear}.  Then, for all $0<t_1\leq t_2<T,$ $x_0\in \Dom $, $r>0$ such that $\mathrm{dist}({B(x_0,r)}, \partial \Dom)>0$,
    \begin{align*}
        v\in C([t_1,t_2]; C^{\infty}({B(x_0,r)};\mathbb{R}^2))\quad \mathbf{P}-a.s.
    \end{align*}
\end{corollary}
\begin{proof}
Since $\mathrm{dist}({B(x_0,r)}, \partial \Dom)>0,\ 0<t_1\leq t_2<T$ we can find $\epsilon$ small enough such that $0<t_1-2\epsilon<t_1\leq t_2<t_2+2\epsilon<T,\ \mathrm{dist}({B(x_0,r+2\epsilon)}, \partial \Dom)>0$ and $\psi\in C^{\infty}_c((0,T)\times\Dom)$ supported in $[t_1-\epsilon,t_2+\epsilon]\times B(x_0, r+\epsilon)$ such that it is equal to one in $[t_1+\epsilon/2,t_2+\epsilon/2]\times B(x_0,r+\epsilon/2)$.
From \autoref{Lemma uniform boundendness} and Sobolev embedding theorem we know that $v\in C([t_1-\epsilon,t_2+\epsilon];L^{\infty}(B(x_0,{r}+\epsilon);\mathbb{R}^2))\ \mathbf{P}-a.s.$
Denoting by
\begin{align*}
    &\omega=\operatorname{curl}v\in C([0,T];H^{-1})\cap L^2((0,T)\times \Dom), \\ & \omega_w=\operatorname{curl}w\in  C([t_1-2\epsilon,t_2+\epsilon], C^{\infty}({B(x_0,r+2\epsilon)}))\quad \mathbf{P}-a.s.
    \end{align*}
    and $\Tilde{\omega}=\omega \psi\in L^2((0,T)\times \mathbb{R}^2)$ supported in $[t_1-\epsilon,t_2+\epsilon]\times B(x_0,r+\epsilon)$, then, arguing as in the proof of \autoref{Lemma uniform boundendness}, it follows that $\Tilde{\omega}$ is a distributional solution in $(0,T)\times B(x_0, r+\epsilon)$ of
\begin{align}\label{distributional solution tilde omega 3}
    \partial_t \tilde{\omega}&= \Delta \tilde{\omega}+\tilde{g}
\end{align}
with \begin{align*}
 \tilde{g} & =-v\cdot \nabla\tilde{\omega}-w\cdot\nabla\tilde{\omega}+\partial_t{\psi} \omega-2\nabla{\psi}\cdot\nabla \omega-\Delta {\psi} \omega+v\cdot\nabla{\psi} \omega\\ &-{\psi} w\cdot\nabla\omega_w-{\psi} v\cdot\nabla \omega_w+w\cdot\nabla{\psi}\omega.
\end{align*}
From the regularity of $\omega,\ v,\ \tilde\omega,\ \omega_w,\ w$, then $\tilde g \in L^2(t_1-\epsilon,t_2+\epsilon;H^{-1}(B(x_0,r+\epsilon)))\ \mathbf{P}-a.s$. By standard regularity theory for the heat equation, see for example Step 2 in \cite[Theorem 13.1]{lemarie2018navier}, a solution of \eqref{distributional solution tilde omega 3} with $\tilde g \in L^2(t_1-\epsilon,t_2+\epsilon;H^{k-1}(B(x_0,r+\epsilon))),\ k\in\mathbb{N}$, belongs to $C([{t}_1-\epsilon/2,{t}_2+\epsilon/2];H^k(B(x_0,{r}+\epsilon/2)))\cap L^2({t}_1-\epsilon/2,{t}_2+\epsilon/2;H^{k+1}(B(x_0,{r}+\epsilon/2)))$. Therefore \begin{align*}
    \tilde \omega &\in C([{t}_1-\epsilon/2,{t}_2+\epsilon/2];L^2(B(x_0,{r}+\epsilon/2)))\\ &\cap L^2({t}_1-\epsilon/2,{t}_2+\epsilon/2;H^{1}(B(x_0,{r}+\epsilon/2)))\quad \mathbf{P}-a.s.
\end{align*} 
which implies 
\begin{align*}
     \omega &\in C([{t}_1+\epsilon/4,{t}_2-\epsilon/4;L^2(B(x_0,{r}+\epsilon/4)))\\ &\cap L^2({t}_1-\epsilon/4,{t}_2+\epsilon/4;H^{1}(B(x_0,{r}+\epsilon/4)))\quad \mathbf{P}-a.s.
\end{align*} 
since $\psi\equiv 1$ on $({t}_1-\epsilon/2,{t}_2+\epsilon/2)\times B(x_0,{r}+\epsilon/2).$ Considering now $\phi\in C^{\infty}_c(\Dom)$ supported on $B(x_0,r+\epsilon/4)$ such that $\phi\equiv 1 $ on $B(x_0,{r}+\epsilon/8)$, since $\operatorname{curl}v=\omega$ then 
$\phi v$ satisfies 
\begin{align}\label{elliptic v}
    \Delta(\phi v)=\nabla^{\perp}\omega \phi+\Delta \phi v+2\nabla\phi\cdot\nabla v,\quad (\phi v)|_{\partial B(x_0,r+\epsilon/4)}=0.  
\end{align}
Since 
\begin{align*}
 \nabla^{\perp}\omega \phi+\Delta \phi v+2\nabla\phi\cdot\nabla v    &\in C([{t}_1+\epsilon/4,{t}_2-\epsilon/4;H^{-1}(B(x_0,{r}+\epsilon/4)))\\ &\cap L^2({t}_1-\epsilon/4,{t}_2+\epsilon/4;L^{2}(B(x_0,{r}+\epsilon/4)))\quad \mathbf{P}-a.s., 
\end{align*}
by standard elliptic regularity theory (see for example \cite{ambrosio2019lectures}), \begin{align*}
    \phi v & \in C([{t}_1+\epsilon/4,{t}_2-\epsilon/4;H^1(B(x_0,{r}+\epsilon/4)))\\ &\cap L^2({t}_1-\epsilon/4,{t}_2+\epsilon/4;H^{2}(B(x_0,{r}+\epsilon/4)))\quad \mathbf{P}-a.s..
\end{align*}
Since $\phi \equiv 1$ on $B(x_0,{r}+\epsilon/8)$ then 
\begin{align*}
v& \in C([{t}_1+\frac{\epsilon}{16},{t}_2-\frac{\epsilon}{16};H^{1}(B(x_0,{r}+\frac{\epsilon}{16})))\\ & \cap L^2({t}_1-\frac{\epsilon}{16},{t}_2+\frac{\epsilon}{16};H^2(B(x_0,{r}+\frac{\epsilon}{16})))\quad \mathbf{P}-a.s.    
\end{align*}
Reiterating the argument as in Step 3 in \cite[Theorem 13.1]{lemarie2018navier} the thesis follows.
\end{proof}

\begin{proof}[Proof of \autoref{interior regularity theorem}]
The claim follows by  \autoref{corollary inteorior regularity linear}, \autoref{corollary inteorior regularity nonlinear} and a localization argument. Moreover, to obtain the claimed smoothness up to time $t=T$, let us consider the extension by $0$ of $h_b$ on $[0,T+1]$, i.e.\
    \begin{equation*}
        \widetilde{h}_b(t)=
        \left\{
        \begin{aligned}
            &h_b(t),\quad & \text{if }\ &t\in (0,T),\\
            &0, & \text{if }\ & t\in (T,T+1).
        \end{aligned}\right.
    \end{equation*} 
    Let $\widetilde{u}$ be the unique weak solution \eqref{Intro equation} provided by \autoref{main thm} with $T$ replaced by $T+1$. Then, by  \autoref{corollary inteorior regularity linear}, \autoref{corollary inteorior regularity nonlinear} and a standard covering argument, for all $t_0\in (0,T)$, $\Dom_0\subset \Dom $ such that $\mathrm{dist}({\Dom_0}, \partial \Dom)>0$,
    \begin{align}
    \label{eq:regularity_u_tilde}
        \widetilde{u}\in C([t_0,T];C^{\infty}(\Dom_0;\mathbb{R}^2))\quad \mathbf{P}-a.s.
    \end{align} 

    Now, let $u$ be the unique weak solution of \eqref{Intro equation} provided by \autoref{main thm}. By uniqueness, we have $u=\widetilde{u}|_{[0,T]}$ and the conclusion follows from \eqref{eq:regularity_u_tilde}.
\end{proof}

\subsection*{Acknowledgements} The authors thank Professor Franco Flandoli for useful discussions and valuable insight into the subject. In particular, A.A. would like to thank professor Franco Flandoli for hosting and financially contributing to his research visit at Scuola Normale di Pisa in January 2023, where this work started.
E.L. would like to express sincere gratitude to Professor Marco Fuhrman for igniting his interest in this particular field of research.  E.L. want to thank Professor Matthias Hieber and Dr. Martin Saal for useful discussions.

\appendix 

\section{$H^{\infty}$-calculus for the Stokes operator}\label{appendinx H infty calulus}
In this appendix we prove \autoref{t:bounded_H_infty}. 
Here we use the transference result proven in \cite{KW17_stokes}. In this section we also need the concept of $\rsec$-sectoriality, which can be again found in \cite[Chapters 3 and 4]{pruss2016moving} and \cite[Chapter 10]{Analysis2}. 

To discuss the main idea in the proof of \autoref{t:bounded_H_infty}, let us define the operator $B_q v:=- \Delta v$ on $L^q(\Dom;\R^2)$ with domain
\begin{align*}
\Do(B_q):= \big\{f=(f_1,f_2)\in W^{2,q}(\Dom;\R^2)\,:\,
&   f|_{\Gamma_b}=0, \\  
& f_2 |_{\Gamma_u}=\partial_z f_1|_{\Gamma_u}=0\big\}.
\end{align*}
Note that $B_q$ has the same boundary conditions of $A_q$. However, $B_q$ is considerably more simple than $A_q$ since, to study its spectral properties, it is possible to use reflection arguments which are not available for the Stokes operator, cf.\ \Cref{ss:laplace_H}.

We prove \autoref{t:bounded_H_infty} by using the transference techniques developed in \cite{KW17_stokes}. By \cite[Theorem 9]{KW17_stokes}, we divide the proof into the following steps:

\begin{enumerate}[leftmargin=*]
\item\label{it:Stokes_H_calculus2} Boundedness of the $H^{\infty}$-calculus for $A_2$ and $B_2$ (i.e.\ in the Hilbertian case).
\item\label{it:Laplacian_H_calculus} Boundedness of the $H^{\infty}$-calculus for $B_q$ for all $q\in (1,\infty)$.
\item\label{it:Stokes_R_sectoriality} $\rsec$-sectoriality of $A_q$.
\item\label{it:transference} Conclusion via transference and interpolation \cite[Theorems 5 and 9]{KW17_stokes}.
\end{enumerate}

\subsection{The Hilbertian case}
Here we analyse the $L^2$-case of \autoref{t:bounded_H_infty}, i.e.\ the operators $A:=A_2$ and $B:=B_2$ acting on $\Ls^2$ and $L^2(\Dom;\R^2)$, respectively. 

\begin{proposition}
\label{prop:self_adjoint} $A$ and $B$ are positive self-adjoint operators, and
\begin{equation}
\label{eq:kato_characterization}
\begin{aligned}
     \Do(A^{\gamma})=\Do((A^*)^{\gamma})=\Hs^{2\gamma}(\Dom),\quad \Do(B^{\gamma})=\Do((B^*)^{\gamma})= H^{2\gamma}(\Dom;\R^2)
\end{aligned}
\end{equation}
for all $\gamma\in (0,\tfrac{1}{4})$.
\end{proposition}

The above result and \cite[Proposition 10.2.23]{Analysis2} imply that $A$ and $B$ have bounded $H^{\infty}$-calculus of angle $0$. 
Below we mainly focus on the operator $A$. The argument to treat $B$ is analogous and simpler.

Consider the elliptic problem associated to $A$, i.e.\
\begin{equation}\label{elliptic problem}
\left\{
\begin{aligned}
-\Delta u+\nabla \pi  & =f, \qquad& \text{ on }& \Dom,\\
\mathrm{div}\, u   &=0, & \text{ on }& \Dom,  \\
u     &=0, & \text{ on }& \Gamma_b,\\
\partial_{z}u_1   &=0,\qquad & \text{ on }& \Gamma_u,\\
u_2     &=0,\qquad& \text{ on }& \Gamma_u.
\end{aligned}\right.
\end{equation}

If $f\in V^*$, the definition of weak solutions $u\in V$ is standard and similar to the one of \eqref{eq:psi_f_problem}. The well-posedness of \eqref{elliptic problem} is analysed below.

\begin{proposition}\label{Proposition well posed L2}
    For each $f\in V^*$ there exists a unique solution of problem \eqref{elliptic problem}. 
    Moreover if $f\in H$ then $u\in \Do(A)$ and
    \begin{align*}
        \lVert u\rVert_{\Do(A)}+\lVert \pi\rVert_{H^1/\mathbb{R}}\leq C\lVert f\rVert.
    \end{align*}
\end{proposition}
\begin{proof}
Existence of weak solutions follows immediately by Lax-Milgram Lemma, since Poincaré inequality holds in $V$. Therefore we can endow $V$ with the norm $\lVert \nabla u\rVert_V:=\lVert \nabla u\rVert_{L^2}$ equivalent to the standard $H^1$ norm. Let now $f\in H$, therefore the weak formulation satisfied by $u$ reduces to 
\begin{align}
    \langle \nabla u,\nabla \phi\rangle=\langle f,\phi\rangle\quad \forall\phi \in V \label{weak formulation 2 auxiliary}.
\end{align}
Considering $\phi\in \mathcal{D}(\Dom)$ it follows that 
\begin{align*}
    \langle \Delta u+f,\phi\rangle_{\mathcal{D}(\Dom)', \mathcal{D}(\Dom)}=0,
\end{align*}
therefore by \cite[Proposition 1.1, Proposition 1.2]{temam2001navier} it exists $\pi \in L^2(D)$ such that 
\begin{align}\label{distribution equation L2}
    -\Delta u+\nabla\pi=f
\end{align} in the sense of distributions and \begin{align*}
    \lVert  \pi\rVert_{L^2/\mathbb{R}}\leq C \lVert \Delta u\rVert_{H^{-1}}+\lVert f\rVert\lesssim \lVert f\rVert.
\end{align*}
The higher regularity follows by the standard Niremberg's method of finite difference quotients.  Therefore,  fix $h>0$, extending periodically either $u$ and $f$ in the $x$ direction and consider $\phi=\tau_{h}\tau_{-h} u$ as a test function in \eqref{weak formulation 2 auxiliary}, where $\tau_h v=\frac{v(x+h,z)-v(x,z)}{h}.$ Then by change of variables it follows that \begin{align*}
    \lVert \tau_{-h}\nabla u\rVert\leq C\lVert f\rVert.
\end{align*}
The relation above implies that 
\begin{align*}
     \lVert \partial_x\nabla u\rVert\leq C\lVert f\rVert.
\end{align*}
Let us now consider $\phi=\partial_x\psi,\ \psi\in \mathcal{D}(\Dom)$ as test function in \eqref{weak formulation 2 auxiliary}. Therefore arguing as above it follows that $\partial_x \pi\in L^2(\Dom)$ and $\lVert \partial_x\pi\rVert_{L^2}\lesssim \lVert f\rVert $. Since equation \eqref{distribution equation L2} is satisfied in the sense of distribution and $u$ is divergence free it follows that \begin{align*}
    \partial_z \pi=f_2+\partial_{xx}u_2-\partial_{xz}u_1\in L^2(\Dom).
\end{align*}
This implies that $\lVert \nabla\pi\rVert_{L^2}\lesssim \lVert f\rVert.$
Lastly $u_1$ satisfies \begin{align*}\label{elliptic equation auxiliary}
        -\partial_{zz}u_1&=-\partial_x\pi+\partial_{xx}u_1+f_1\in L^2(\Dom),    
\end{align*}
which completes the proof.
\end{proof}

We are ready to prove \autoref{prop:self_adjoint}.

\begin{proof}[Proof of \autoref{prop:self_adjoint}]
\emph{Step 1: $A$ and $B$ are a positive self-adjoint operators}. As above we only discuss the operator $A$. The positivity of $A$ is clear. Next, note that, integrating by parts
\begin{align*}
    \langle A u,v\rangle=\langle u,A v\rangle\quad \forall u,v\in \Do(A).
\end{align*}  
This means that $A$ is symmetric.
It remains to show that $\Do(A^*)=\Do(A)$ and $\forall u\in \Do(A),\ A^*u=Au$.
By definition \begin{align*}
    \Do(A^*)=\{u \in H:\quad & F: D(A)\subseteq H\rightarrow \mathbb R,\quad F(v)=\langle u,Av\rangle\\ & \text{ has a linear bounded extension on } H \}.
    \end{align*} For each $u\in \Do(A^*),\ F(v)=\langle u,Av\rangle=\langle f_u,v\rangle$ therefore $A^*u=f_u$. In particular, $\forall u\in \Do(A^*)\ \langle u,Av\rangle=\langle A^*u,v\rangle$. Thanks to the fact that $A$ is symmetric we have $\Do(A)\subseteq \Do(A^*)$. Given now $v\in \Do(A^*),\ f_v=A^*v\in H$,  let us consider the boundary value problem \eqref{elliptic problem} with forcing term equal to $f_v$. By \autoref{Proposition well posed L2} it has a unique solution $(w,\pi)\in \Do(A)\times H^1(\Dom)$, this implies that $Aw=f_v=A^*v$. For each $z\in H$, let us consider the boundary value problem \eqref{elliptic problem} with forcing term equal to $z$. By \autoref{Proposition well posed L2} it exists a unique $S_z\in \Do(A)$ such that $As_z=z$. Therefore $\langle z,w-v\rangle=\langle As_z,w-v\rangle=\langle s_z,Aw-A^*v\rangle=0$ thanks to the fact that $A$ is symmetric. Since $z$ is arbitrary, then $v=w$ and the claim follows.

\emph{Step 2: Proof of \eqref{eq:kato_characterization}}. We begin by proving the first identity in \eqref{eq:kato_characterization}. Note that $\Do(B^{\gamma})=\Do((B^*)^{\gamma})$ for $\gamma <1/2$ follows from \cite[Theorem 1.1]{K61} and Step 1.  
By Step 1 and \cite[Proposition 10.2.23]{Analysis2}, $B$ has bounded $H^{\infty}$-calculus and in particular $B$ has the bounded imaginary powers property, \cite[Subsection 3.4]{pruss2016moving}. By \cite[
Theorem 3.3.7]{pruss2016moving}, 
$
\Do(B^{\gamma})=[L^2(\Dom;\R^2),\Do(B)]_{\gamma}
$ for all $\gamma<1$. The latter gives $
\Do(B^{\gamma})=H^{2\gamma}(\Dom;\R^2)$ in case $\gamma<1/4$ by \cite{Se72}. The second identity in  \eqref{eq:kato_characterization} follows analogously, where one uses the argument in \cite{FM70} (see also \cite[Proposition 5.5, Chapter 17]{TayPDE3}) to deduce 
$\Do(A^{\gamma})=\Hs^{2\gamma}(\Dom)$ from the first identity in 
\eqref{eq:kato_characterization}.
\end{proof}

\subsection{Bounded $H^{\infty}$-calculus for Laplace operators}
\label{ss:laplace_H}
In this subsection we prove the boundedness of the $H^{\infty}$-calculus for $B_q$. The basic idea is to use the product structure of the domain $\Dom$ and to write $B_{q}u=(\Lr u_1,\Ld u_2)$ where
\begin{align*}
\Do(\Ld)&:=\big\{f\in W^{2,q}(\Dom)\,:\, f|_{\Gamma_b}=f|_{\Gamma_u}=0 \big\}, & & \Ld f:= \Delta f,\\
\Do(\Lr)&:=\big\{f\in W^{2,q}(\Dom)\,:\, \partial_z f|_{\Gamma_b}=0,\  f|_{\Gamma_u}=0 \big\}, & & \Lr f:= \Delta f.
\end{align*}

\begin{proposition}[Bounded $H^{\infty}$-calculus for Laplace operators]
Let $q\in (1,\infty)$ and let $\Dom$ be as above. 
Then $-\Ld$ and $-\Lr$ have a bounded $H^{\infty}$-calculus of angle $0$.
In particular $B_q$ has a bounded $H^{\infty}$-calculus of angle $0$. 
\end{proposition}

The above statement also holds for the Neumann Laplacian, but it will not be needed below.

\begin{proof}
We divide the proof into three steps. In the first step, we exploit the product structure of our domain to reduce the problem to a one dimensional situation.

\emph{Step 1: Reduction to the 1d case}. 
Then the Dirichlet and the Robin Laplacian in 1d are given by
\begin{align*}
\Do(\Ldd)&:=\big\{f\in W^{2,q}(0,a)\,:\, f(0)=f(a)=0 \big\}, & & \Lrd f:= \partial_x^2 f,\\
\Do(\Lrd)&:=\big\{f\in W^{2,q}(0,a)\,:\, \partial_x f(0)= f(a)=0 \big\}, & & \Lrd f:= \partial_x^2 f.
\end{align*}
Let us consider $\Ldd$, the other case is analogue. In this step we assume that $-\Ldd$ has a bounded $H^{\infty}$-calculus of angle $0$.
Let $\Lt$ be the Laplacian on the periodic torus $\T$ with domain $W^{2,q}(\T)$. The boundedness of the $H^{\infty}$-calculus for $-\Lt$ such operator follows from the periodic version of \cite[Theorem 10.2.25]{Analysis2} and $\angH(\Lt)=0$. 

On $L^2(\Dom)$ considers the operator
\begin{align*}
(\Ldd^{(x)}f)(x,z)= (\Ldd f(\cdot,z))(x), \qquad
(\Lt^{(y)}f)(x,z)= (\Ldd f(x,\cdot))(z),    
\end{align*}
with the corresponding natural domains. It is clear that both $-\Ldd^{(x)}$ and $-\Lt^{(z)} $ have bounded $H^{\infty}$-calculus of angle $0$. Now by sum of commuting operators \cite[Corollary 4.5.8]{pruss2016moving}, the sum operator
$
-A_q:=-\Ldd^{(x)}-\Lt^{(z)}
$
has a bounded $H^{\infty}$-calculus of angle $0$ with domain
\begin{equation*}
\Do(A_q)=\Do(\Ldd^{(x)}) \cap \Do(\Lt^{(z)})= \Do(\Ld)    
\end{equation*}
where the last equality follows from elliptic regularity.

\emph{Step 2: $-\Ld$ has a bounded $H^{\infty}$-calculus of angle $0$.}
By rescaling and translation we may replace $(0,a)$  by $(-\pi,\pi)$. Let $\Lt$ be the Laplacian on the periodic torus $\T=(-\pi,\pi)$ (as measure space) with domain $W^{2,q}(\T)$. Let 
\begin{equation*}
Y:=\big\{f\in L^2(\T)\,:\, f=\sum_{n\geq 0} f_n \sin(n x) \text{ where } (f_n)_{n\geq 1} \in \ell^2  \big\}.    
\end{equation*}
It is clear that $Y\subseteq L^2(\T)$ is closed, and 
\begin{equation*}
(\lambda- \Lt)^{-1}: Y\to Y\ \ \text{ for all }\ \ \lambda\in \rho(\Ld).    
\end{equation*}
Now note that $\Ld$ is the part of $\Lt$ on $Y$, i.e.\
\begin{align*}
\Do(\Ld)
&=\{f\in \Do(\Lt) \cap Y\,:\, \Lt f\in Y\}, \\
 \Ld f
 &= \Lt f \ \ \text{ for all }f \in \Do(\Ld).
\end{align*}
Now the claim of Step 1 follows from \cite[Proposition 10.2.18]{Analysis2} and the periodic version of \cite[Theorem 10.2.25]{Analysis2}.

\emph{Step 2: $-\Lr$ has a bounded $H^{\infty}$-calculus of angle $0$.}
As in the above step, by rescaling we replace $(0,a)$ by $(0,\pi)$. Consider the reflection map
\begin{align*}
 Rf(z):=
\begin{cases}
 f(z) \quad &z\in (0,\pi),\\
 f(-z)\quad &z\in (-\pi,0).
\end{cases}    
\end{align*}
Let $\Ld$ be the Dirichlet Laplacian on $(-\pi,\pi)$. Then one can readily check that $\rho(\Ld)\subseteq \rho(\Lr)$ and for all $\lambda\in \rho(\Ld)$
\begin{equation*}
(\lambda-\Lr)^{-1} f =\big[(\lambda- \Ld)^{-1} Rf \big] \Big|_{(0,\pi)}.    
\end{equation*}
Now the claim follows from Step 1 and the definition of $H^{\infty}$-calculus. 
\end{proof}

\subsection{$\rsec$-sectoriality for the Stokes operator}
\label{ss:maximal_regularity_stokes}
For the notion of $\rsec$-boundedness of a family of linear operators we refer to \cite[Chapter 8]{Analysis2}. For a family of linear operators $\mathscr{J}$, the $\rsec$-bound is denoted by $\rsec(\mathscr{J})$. 
As in \cite[Chapter 10]{Analysis2}, we said that operator $T$ on a Banach space $X$ is said to be $\rsec$-sectorial if there exists $\phi\in (0,\pi)$ such that $\rho(A)\subseteq \{\lambda\in \C\,|\, |\arg \lambda|\geq \pi-\phi \}$ and 
\begin{equation*}
\rsec(\lambda (\lambda-T)^{-1}\,|\, |\arg \lambda|>\pi-\phi )<\infty.    
\end{equation*}
The $\rsec$-sectoriality angle is the infimum over all $\phi\in (0,\pi)$ for which the above holds.
The main result of this subsection reads as follows. 

\begin{proposition}
\label{prop:max_reg_stokes}
For all $q\in (1,\infty)$, the operator $A_q$ is $\rsec$-sectorial with $\rsec$-sectoriality angle $<\pi/2$.
\end{proposition}

\begin{proof}
Fix $q\in (1,\infty)$. For simplicity we first prove the statement for a shifted Stokes operator and in a second step we conclude by a simple translation argument. 

\emph{Step 1: There exists $\lambda_q$ such that $\lambda_q+A_q$ is $\rsec$-sectorial with $\rsec$-sectoriality angle $<\pi/2$.} 
Due to the the well-known equivalence of maximal $L^q$-regularity and $\rsec$-sectoriality proven by L.\ Weis \cite{We} (see also \cite[Subsection 4.2, Chapter 3]{pruss2016moving}), it is enough to show that, for all $f\in L^q(0,1;\Ls^q)$, the Stokes problem on $\Dom$,
\begin{equation}
\label{eq:linear_stokes}
\left\{
\begin{aligned}
\partial_t u & =\Delta u-\nabla P +f, &\ \ \ \text{ on }&(0,1)\times\Dom,\\
\mathrm{div}\, u&=0,	&\text{ on }&(0,1)\times\Dom,	\\
u&=0, & \text{ on }&(0,1)\times\Gamma_b, \\  
 u_2&= \partial_z u_1=0,& \text{ on } &(0,1)\times\Gamma_u,\\
 u(0)&=0,	&\text{ on }&\Dom,
\end{aligned}
\right.
\end{equation}
admits a unique solution in the class 
\begin{equation}
\label{eq:max_reg_class}
u\in W^{1,q}(0,1;\Ls^q)\cap L^q(0,1;\Ws^{2,q}(\Dom)), \quad P\in  L^q(0,1;W^{1,q}(\Dom)).
\end{equation}

The proof follows a standard localization argument. Let $(\phi_j)_{j=1}^N$ be a smooth partition of the unity such that, for all $j\in \{1,\dots,N\}$, $\mathrm{diam}(\mathrm{supp}\,\phi_k)<\frac{1}{2}$
\begin{equation*}
\text{ either }\quad  \mathrm{supp}\,\phi_j \cap (\Tor\times \{0\})=\varnothing \quad \text{ or }\quad 
  \mathrm{supp}\,\phi_j\cap (\Tor\times \{a\})=\varnothing.    
\end{equation*}

Fix $k\in \{1,\dots,N\}$. 
Multiplying \eqref{eq:linear_stokes} by $\phi_k$, we obtain for $u_{k}:=\phi_k u$ and $P_k=\phi_k P$ either
\begin{equation}
\label{eq:linear_stokes_1}
\left\{
\begin{aligned}
\partial_t u_k  &=\Delta u_k-\nabla P_k +\phi_k f + \mathcal{L}_k (u,P), & \text{ on }&(0,1)\times\R	\times(0,\infty),\\
\mathrm{div}\, u_k&=\nabla\phi_k \cdot u, & \text{ on }&(0,1)\times \R\times(0,\infty),	\\
u_k&=0, &  \text{ on }&(0,1)\times\R\times \{0\}, \\  
 u_k(0)&=0,	& \text{ on }&\R\times(0,\infty),
\end{aligned}
\right.
\end{equation}
or
\begin{equation}
\label{eq:linear_stokes_2}
\left\{
\begin{aligned}
\partial_t u_k  &=\Delta u_k-\nabla P_k +\phi_k f+  \mathcal{L}_k (u,P), & \text{ on }&(0,1)\times \R\times (-\infty,a),\\
\mathrm{div}\, u_k&=\nabla \phi_k \cdot u,	&\text{ on }&(0,1)\times\R\times (-\infty,a),\\	
 \phi_k u_2&=0, & \text{ on } &(0,1)\times\R\times \{a\},\\
\partial_z (\phi_k u_1)&= u_1 \partial_z \phi_k,& \text{ on } &(0,1)\times\R\times \{a\},\\
 u_k(0)&=0,	&\text{ on }&\R\times (-\infty,a).
\end{aligned}
\right.
\end{equation}
Here $\mathcal{L}_k$ denotes a lower order operator w.r.t.\ to the maximal regularity space for $(u,P)$ in \eqref{eq:max_reg_class}. 

Maximal $L^p(L^q)$-regularity estimates for \eqref{eq:linear_stokes_1} and \eqref{eq:linear_stokes_2} are proven in \cite[Theorem 7.2.1]{pruss2016moving} in the case of no-slip or pure-slip, respectively (see conditions (7.16) and (7.17) on \cite[p.\ 323]{pruss2016moving}). 
Now a-priori estimates for solutions as in \eqref{eq:max_reg_class} in the maximal $L^p(L^q)$-regularity class $W^{1,q}(0,1;\Ls^q)\cap L^q(0,1;\Ws^{2,q}(\Dom))$ follows by repeating the localization argument of \cite[Subsection 3.4 in Chapter 7]{pruss2016moving} to adsorb the lower order terms. 

It remains to discuss the existence of solutions as in \eqref{eq:max_reg_class}. Arguing as in step 3 of \autoref{Regularity of the Neumann map} and using $L^2$-theory, one can prove existence of smooth solutions to equation \eqref{eq:linear_stokes} in case of smooth data $f$. Hence the existence follows from a standard density argument and the a-priori estimates obtained above for solutions in the class $W^{1,q}(0,1;\Ls^q)\cap L^q(0,1;\Ws^{2,q}(\Dom))$.

\emph{Step 2: Conclusion}. By Step 1, it remains to remove the shift $\lambda_q$. 
Arguing as in \cite[Proposition 2.2]{AH21}, by holomorphicity of the resolvent and \cite[Proposition 4.1.12]{pruss2016moving}, it is enough to show that 
\begin{equation*}
\rho(A_q)\subseteq \{\lambda\in \C\,|\, |\arg z|>\psi\}, \quad \text{ for some }\psi<\pi/2.    
\end{equation*}

In the case $q>2$, noticing that $(\lambda- A_q) = (\lambda-\lambda_q-A_q) +\lambda_q $ and that $\rho(\lambda_q+A_q)\subseteq \{\lambda\in \C\,|\, |\arg z|>\phi\}$ for some $\phi<\pi/2$ by Step 1, the conclusion can be obtained by using a standard bootstrap method via Sobolev embeddings. 

In the case $q<2$ one uses $(A_q)^*=A_{q'}$.
\end{proof}

\subsection{Proof of \autoref{t:bounded_H_infty}}
As the proof of \autoref{t:bounded_H_infty} follows the one of \cite[Theorem 16]{KW17_stokes}, we only provide a sketch.

\begin{proof}[Proof of \autoref{t:bounded_H_infty} - Sketch]
\emph{Step 1: There exists $\beta>0$ for which the following estimates hold}:
\begin{equation}
\label{eq:decay_littlewood_paley}
\begin{aligned}
\sup_{1\leq t,s\leq 2} \rsec(\varphi(2^{j+\ell } s A_2) \psi (t2^{j} B_2))&\lesssim 2^{-\beta \ell},\\
\sup_{1\leq t,s\leq 2} \rsec(\varphi(2^{j+\ell }s A_2)^*\p_{p}^*\psi (t2^{j} B_2)^* )&\lesssim 2^{-\beta \ell}.
\end{aligned}
\end{equation}

Recall that $\rsec(\mathscr{J})$ stand for the $\rsec$-bound of the family of operators $\mathscr{J}$, see \cite[Chapter 9]{Analysis2} for details on $\rsec$-boundedness.

By elliptic regularity we have $\p: H^{1}(\Dom;\R^2)\to \Hs^{1}(\Dom)$. Interpolating we obtain 
$\p: H^{s}(\Dom;\R^2)\to \Hs^{s}(\Dom;\R^2)$ for all $s\in (0,1)$. Hence $\p: \Do(B^{\gamma})\to \Do(A^{\gamma})$ for all
$\gamma\in (0,1/4)$. The estimate \eqref{eq:decay_littlewood_paley} now follows from \cite[Proposition 10]{KW17_stokes} and \eqref{eq:kato_characterization}.

\emph{Step 2: Boundedness of the $H^{\infty}$-calculus}. Next we argue as in the proof of \cite[Theorem 5]{KW17_stokes}. 
Let $q\in (1,\infty)$ be as in the statement of \autoref{t:bounded_H_infty} and fix $p\in (q,\infty)$.
By $\rsec$-sectoriality of $A_p$ and $B_p$ (i.e. \autoref{prop:max_reg_stokes}) and \cite[Proposition 10.3.2]{Analysis2}, 
\begin{equation}
\label{eq:rsec_Ap}
\rsec (\varphi(t A_p)\,:\, t>0 )\leq c_0  \quad \text{ and }\quad  \rsec(\psi(s B_p)\,:\, s>0)\leq c_0.
\end{equation}
Note that $(A_r)_{r\in (1,\infty)}$, $(B_r)_{r\in (1,\infty)}$ are consistent family of operators. Hence, by complex interpolation and \cite[Proposition 8.4.4]{Analysis2}, we have that \eqref{eq:decay_littlewood_paley} holds for some $\beta=\beta(r,p)>0$ and with $(A_2,B_2)$ replaced by $(A_q,B_q)$. 
Now the boundedness of the $H^{\infty}$-calculus follows from \cite[Theorem 9]{KW17_stokes}.

\emph{Step 3: Description of the fractional powers}. To obtain the description of the fractional powers of $A_q$ and $B_q$ one can argue as in the proof of \eqref{eq:kato_characterization} by using the bounded imaginary power property and \cite{FM70,Se72}.
\end{proof}

\bibliography{biblio}{}

\begin{thebibliography}{10}

\bibitem{AH21}
Antonio Agresti and Amru Hussein.
\newblock Maximal {$L^p$}--regularity and {$H^\infty$}--calculus for block operator matrices and applications.
\newblock {\em J. Funct. Anal.}, 285(11):Paper No. 110146, 2023.

\bibitem{ALV23}
Antonio Agresti, Nick Lindemulder, and Mark Veraar.
\newblock On the trace embedding and its applications to evolution equations.
\newblock {\em Mathematische Nachrichten}, 296(4):1319--1350, 2023.

\bibitem{AV20}
Antonio Agresti and Mark Veraar.
\newblock Stability properties of stochastic maximal ${L}^p$--regularity.
\newblock {\em Journal of Mathematical Analysis and Applications}, 482(2):123553, 2020.

\bibitem{AV22_QSEE1}
Antonio Agresti and Mark Veraar.
\newblock Nonlinear parabolic stochastic evolution equations in critical spaces part {I}. {S}tochastic maximal regularity and local existence.
\newblock {\em Nonlinearity}, 35(8):4100, 2022.

\bibitem{alos2002stochastic}
Elisa Al{\`o}s and Stefano Bonaccorsi.
\newblock Stochastic partial differential equations with {D}irichlet white-noise boundary conditions.
\newblock {\em Annales de l'IHP Probabilit{\'e}s et statistiques}, 38(2):125--154, 2002.

\bibitem{Amann95}
Herbert Amann.
\newblock {\em Linear and quasilinear parabolic problems}, volume~1.
\newblock Springer, 1995.

\bibitem{ambrosio2019lectures}
Luigi Ambrosio, Alessandro Carlotto, and Annalisa Massaccesi.
\newblock {\em Lectures on elliptic partial differential equations}, volume~18.
\newblock Springer, 2019.

\bibitem{parameter1}
Sigurd Angenent.
\newblock Nonlinear analytic semiflows.
\newblock {\em Proc. Roy. Soc. Edinburgh Sect. A}, 115(1-2):91--107, 1990.

\bibitem{parameter2}
Sigurd Angenent.
\newblock Parabolic equations for curves on surfaces. {I}. {C}urves with {$p$}-integrable curvature.
\newblock {\em Ann. of Math. (2)}, 132(3):451--483, 1990.

\bibitem{BeLo}
J{\"o}ran Bergh and J{\"o}rgen L{\"o}fstr{\"o}m.
\newblock {\em Interpolation spaces: an introduction}, volume 223.
\newblock Springer Science \& Business Media, 2012.

\bibitem{berselli2006existence}
Luigi~C Berselli and Marco Romito.
\newblock On the existence and uniqueness of weak solutions for a vorticity seeding model.
\newblock {\em SIAM journal on mathematical analysis}, 37(6):1780--1799, 2006.

\bibitem{bessaih2014homogenization}
Hakima Bessaih, Yalchin Efendiev, and Florin Maris.
\newblock Homogenization of the evolution {S}tokes equation in a perforated domain with a stochastic {F}ourier boundary condition.
\newblock {\em Netw. Heterog. Media}, 10(2):343--367, 2015.

\bibitem{bessaih2016homogenization}
Hakima Bessaih and Florian Maris.
\newblock Homogenization of the stochastic {N}avier--{S}tokes equation with a stochastic slip boundary condition.
\newblock {\em Applicable Analysis}, 95(12):2703--2735, 2016.

\bibitem{binz2020primitive}
Tim Binz, Matthias Hieber, Amru Hussein, and Martin Saal.
\newblock The primitive equations with stochastic wind driven boundary conditions.
\newblock {\em arXiv preprint arXiv:2009.09449}, 2020.

\bibitem{Zanella}
Stefano Bonaccorsi and Margherita Zanella.
\newblock Absolute continuity of the law for solutions of stochastic differential equations with boundary noise.
\newblock {\em Stoch. Dyn.}, 17(6):1750045, 31, 2017.

\bibitem{brzezniak2015second}
Zdzis{\l}aw Brze{\'z}niak, Ben Goldys, Szymon Peszat, and Francesco Russo.
\newblock Second order {PDE}s with {D}irichlet white noise boundary conditions.
\newblock {\em Journal of Evolution Equations}, 15(1):1--26, 2015.

\bibitem{da2002two}
Giuseppe Da~Prato and Arnaud Debussche.
\newblock Two-dimensional {N}avier--{S}tokes equations driven by a space--time white noise.
\newblock {\em Journal of Functional Analysis}, 196(1):180--210, 2002.

\bibitem{da1993evolution}
Giuseppe Da~Prato and Jerzy Zabczyk.
\newblock Evolution equations with white--noise boundary conditions.
\newblock {\em Stochastics: An International Journal of Probability and Stochastic Processes}, 42(3-4):167--182, 1993.

\bibitem{da1996ergodicity}
Giuseppe Da~Prato and Jerzy Zabczyk.
\newblock {\em Ergodicity for infinite dimensional systems}, volume 229.
\newblock Cambridge University Press, 1996.

\bibitem{da2014stochastic}
Giuseppe Da~Prato and Jerzy Zabczyk.
\newblock {\em Stochastic equations in infinite dimensions}.
\newblock Cambridge university press, 2014.

\bibitem{dalibard2009asymptotic}
Anne-Laure Dalibard.
\newblock Asymptotic behavior of a rapidly rotating fluid with random stationary surface stress.
\newblock {\em SIAM journal on mathematical analysis}, 41(2):511--563, 2009.

\bibitem{dalibard2009mathematical}
Anne-Laure Dalibard and Laure Saint-Raymond.
\newblock Mathematical study of rotating fluids with resonant surface stress.
\newblock {\em Journal of Differential Equations}, 246(6):2304--2354, 2009.

\bibitem{debussche2007optimal}
Arnaud Debussche, Marco Fuhrman, and Gianmario Tessitore.
\newblock Optimal control of a stochastic heat equation with boundary-noise and boundary--control.
\newblock {\em ESAIM: Control, Optimisation and Calculus of Variations}, 13(1):178--205, 2007.

\bibitem{desjardins1999homogeneous}
Beno{\^\i}t Desjardins and Emmanuel Grenier.
\newblock On the homogeneous model of wind-driven ocean circulation.
\newblock {\em SIAM Journal on Applied Mathematics}, 60(1):43--60, 1999.

\bibitem{fabbri2009lq}
Giorgio Fabbri and Beniamin Goldys.
\newblock An {LQ} problem for the heat equation on the halfline with {D}irichlet boundary control and noise.
\newblock {\em SIAM journal on control and optimization}, 48(3):1473--1488, 2009.

\bibitem{flandoli2022heat}
Franco Flandoli and Eliseo Luongo.
\newblock Heat diffusion in a channel under white noise modeling of turbulence.
\newblock {\em Mathematics in Engineering}, 4(4):1--21, 2022.

\bibitem{flandoli2023stochastic}
Franco Flandoli and Eliseo Luongo.
\newblock {\em Stochastic partial differential equations in fluid mechanics}, volume 2330 of {\em Lecture Notes in Mathematics}.
\newblock Springer, Singapore, 2023.

\bibitem{FM70}
Hiroshi Fujita and Hiroko Morimoto.
\newblock On fractional powers of the {S}tokes operator.
\newblock {\em Proceedings of the Japan Academy}, 46(10Supplement):1141--1143, 1970.

\bibitem{G11_book}
Giovanni Galdi.
\newblock {\em An introduction to the mathematical theory of the {N}avier--{S}tokes equations: {S}teady-state problems}.
\newblock Springer Science \& Business Media, 2011.

\bibitem{gill1982atmosphere}
Adrian~E Gill.
\newblock {\em Atmosphere--ocean dynamics}, volume~30.
\newblock Academic press, 1982.

\bibitem{Goldys23}
Ben Goldys and Szymon Peszat.
\newblock Linear parabolic equation with {D}irichlet white noise boundary conditions.
\newblock {\em J. Differential Equations}, 362:382--437, 2023.

\bibitem{Analysis2}
Tuomas Hyt\"onen, Jan van Neerven, Mark Veraar, and Lutz Weis.
\newblock {\em Analysis in {B}anach spaces. {V}ol. {II}. {P}robabilistic {M}ethods and {O}perator {T}heory.}, volume~67 of {\em Ergebnisse der Mathematik und ihrer Grenzgebiete. 3. Folge.}
\newblock Springer, 2017.

\bibitem{Ju17}
Ning Ju.
\newblock On $h^2$-solutions and $z$-weak solutions of the {3D} primitive equations.
\newblock {\em Indiana University Mathematics Journal}, pages 973--996, 2017.

\bibitem{K61}
Tosio Kato.
\newblock Fractional powers of dissipative operators.
\newblock {\em Journal of the Mathematical Society of Japan}, 13(3):246--274, 1961.

\bibitem{KY13_LaxMilgram}
Hideo Kozono and Taku Yanagisawa.
\newblock Generalized {L}ax-{M}ilgram theorem in {B}anach spaces and its application to the elliptic system of boundary value problems.
\newblock {\em Manuscripta Math.}, 141(3-4):637--662, 2013.

\bibitem{KW17_stokes}
Peer~Christian Kunstmann and Lutz Weis.
\newblock New criteria for the {$H^\infty$}--calculus and the {S}tokes operator on bounded {L}ipschitz domains.
\newblock {\em J. Evol. Equ.}, 17(1):387--409, 2017.

\bibitem{lemarie2018navier}
Pierre~Gilles Lemari{\'e}-Rieusset.
\newblock {\em The {N}avier--{S}tokes problem in the 21st century}.
\newblock CRC press, 2018.

\bibitem{lions1993models}
Jacques-Louis Lions, Roger Temam, and Shou~Hong Wang.
\newblock Models for the coupled atmosphere and ocean,(cao i, ii).
\newblock {\em Comput. Mech. Adv.}, 1:3--119, 1993.

\bibitem{lions1996mathematical}
Pierre-Louis Lions.
\newblock {\em Mathematical {T}opics in {F}luid {M}echanics: {V}olume 1: {I}ncompressible {M}odels}, volume~1.
\newblock Clarendon Press Oxford, 1996.

\bibitem{LV21}
Emiel Lorist and Mark Veraar.
\newblock Singular stochastic integral operators.
\newblock {\em Analysis \& PDE}, 14(5):1443--1507, 2021.

\bibitem{pedlosky1996ocean}
Joseph Pedlosky.
\newblock {\em Ocean circulation theory}.
\newblock Springer Science \& Business Media, 1996.

\bibitem{pedlosky2013geophysical}
Joseph Pedlosky.
\newblock {\em Geophysical fluid dynamics}.
\newblock Springer Science \& Business Media, 2013.

\bibitem{pruss2016moving}
Jan Pr\"{u}ss and Gieri Simonett.
\newblock {\em Moving interfaces and quasilinear parabolic evolution equations}, volume 105 of {\em Monographs in Mathematics}.
\newblock Birkh\"{a}user/Springer, 2016.

\bibitem{Saw_Besov}
Yoshihiro Sawano.
\newblock {\em Theory of {B}esov spaces}, volume~56 of {\em Developments in Mathematics}.
\newblock Springer, Singapore, 2018.

\bibitem{ST87}
Hans-J\"{u}rgen Schmeisser and Hans Triebel.
\newblock {\em Topics in {F}ourier analysis and function spaces}.
\newblock Wiley, 1987.

\bibitem{Se72}
Robert~Thomas Seeley.
\newblock Interpolation in ${L}^{p}$ with boundary conditions.
\newblock {\em Studia Mathematica}, 44(1):47--60, 1972.

\bibitem{serrin1961interior}
James Serrin.
\newblock {\em On the interior regularity of weak solutions of the {N}avier--{S}tokes equations}.
\newblock Mathematics Division, Air Force Office of Scientific Research, 1961.

\bibitem{TaylorPDE3}
Michael~E. Taylor.
\newblock {\em Partial differential equations {I}. {B}asic theory}, volume 115 of {\em Applied Mathematical Sciences}.
\newblock Springer, New York, second edition, 2011.

\bibitem{TayPDE3}
Michael~E. Taylor.
\newblock {\em Partial differential equations {III}. {N}onlinear equations}, volume 117 of {\em Applied Mathematical Sciences}.
\newblock Springer, New York, second edition, 2011.

\bibitem{temam1995navier}
Roger Temam.
\newblock {\em Navier--{S}tokes equations and nonlinear functional analysis}.
\newblock SIAM, 1995.

\bibitem{temam2001navier}
Roger Temam.
\newblock {\em Navier--{S}tokes equations: theory and numerical analysis}, volume 343.
\newblock American Mathematical Soc., 2001.

\bibitem{VNeerven}
Jan van Neerven, Mark Veraar, and Lutz Weis.
\newblock {Stochastic maximal ${L}^p$--regularity}.
\newblock {\em The Annals of Probability}, 40(2):788 -- 812, 2012.

\bibitem{NVVW15_survey}
Jan van Neerven, Mark Veraar, and Lutz Weis.
\newblock Stochastic integration in {B}anach spaces--a survey.
\newblock In {\em Stochastic Analysis: A Series of Lectures: Centre Interfacultaire Bernoulli, January--June 2012, Ecole Polytechnique F{\'e}d{\'e}rale de Lausanne, Switzerland}, pages 297--332. Springer, 2015.

\bibitem{We}
Lutz Weis.
\newblock Operator--valued {F}ourier multiplier theorems and maximal {$L\sb p$}-regularity.
\newblock {\em Math. Ann.}, 319(4):735--758, 2001.

\end{thebibliography}
\bibliographystyle{plain}

\end{document}